\documentclass{siamart1116}
\usepackage{amsmath}
\usepackage{amssymb}
\usepackage{enumerate}
\usepackage{graphicx}
\usepackage{epstopdf}
\usepackage{enumitem}

\usepackage{pgfplots, pgfplotstable, booktabs, colortbl, array}

\DeclareMathOperator*{\argmin}{arg\,min}

\usepackage{mathtools}
\DeclarePairedDelimiter{\ceil}{\lceil}{\rceil}
\DeclarePairedDelimiter{\floor}{\lfloor}{\rfloor}

\begin{document}

\title{High-Order Retractions on Matrix Manifolds using Projected Polynomials}

\author{
  Evan S. Gawlik\thanks{Department of Mathematics, University of California, San Diego
    (\email{egawlik@ucsd.edu}, \email{mleok@math.ucsd.edu}).}
  \and
  Melvin Leok\footnotemark[1]
}

\date{}

\headers{High-Order Retractions on Matrix Manifolds}{E. S. Gawlik and M. Leok}

\maketitle

\begin{abstract}
We derive a family of high-order, structure-preserving approximations of the Riemannian exponential map on several matrix manifolds, including the group of unitary matrices, the Grassmannian manifold, and the Stiefel manifold. Our derivation is inspired by the observation that if $\Omega$ is a skew-Hermitian matrix and $t$ is a sufficiently small scalar, then there exists a polynomial of degree $n$ in $t\Omega$ (namely, a Bessel polynomial) whose polar decomposition delivers an approximation of $e^{t\Omega}$ with error $O(t^{2n+1})$.  We prove this fact and then leverage it to derive high-order approximations of the Riemannian exponential map on the Grassmannian and Stiefel manifolds.  Along the way, we derive related results concerning the supercloseness of the geometric and arithmetic means of unitary matrices.
\end{abstract}

%\begin{keywords}
%Matrix manifold, retraction, geodesic, Grassmannian, Stiefel manifold, Riemannian exponential, matrix exponential, polar decomposition, unitary group, geometric mean, Karcher mean 
%\end{keywords}
%\begin{AMS}
%65F50, 53B20, 51F25, 15M15, 47A64
%\end{AMS}

\section{Introduction}

Approximating the Riemannian or Lie-theoretic exponential map on a matrix manifold is a task of importance in a variety of applications, including numerical integration on Lie groups~\cite{hairer2006geometric,iserles2000lie,hall2014lie,bou2009hamilton}, optimization on manifolds~\cite{absil2009optimization,edelman1998geometry,adler2002newton,lundstrom2002adaptive}, interpolation of manifold-valued data~\cite{sander2012geodesic,sander2015geodesic,grohs2013quasi,gawlik2016interpolation}, rigid body simulation~\cite{bou2009hamilton,kobilarov2009lie}, fluid simulation~\cite{gawlik2011geometric}, and computer vision~\cite{turaga2011statistical,lui2012advances,gallivan2003efficient}.  Often, special attention is paid to preserving the structure of the exponential map~\cite{celledoni2000approximating}, which, for instance, should return a unitary matrix when the input $\Omega$ is skew-Hermitian.  
In this paper, we construct structure-preserving approximations to the Riemannian exponential map on matrix manifolds using \emph{projected polynomials}
 -- polynomial functions of matrices which, when projected onto a suitable set, deliver approximations to the Riemannian exponential with a desired order of accuracy.  These projected polynomials can be thought of as high-order generalizations of the ``projection-like retractions'' considered in~\cite{absil2012projection}.
The matrix manifolds we consider are:
\begin{enumerate}
\item The group of unitary $m \times m$ matrices.
\item The Grassmannian manifold $Gr(p,m)$, which consists of all $p$-dimensional linear subspaces of $\mathbb{C}^m$, where $m \ge p$.
\item The Stiefel manifold $St(p,m) = \{Y \in \mathbb{C}^{m \times p} \mid Y^*Y=I\}$, where $m \ge p$.
\end{enumerate}

The projector we use is to accomplish this task is the map which sends a full-rank matrix $A \in \mathbb{C}^{m \times p}$ ($m \ge p$) to the nearest matrix with orthonormal columns.  The latter matrix is precisely the factor $U$ in the polar decomposition $A=UH$, where $U \in \mathbb{C}^{m \times p}$ has orthonormal columns and $H \in \mathbb{C}^{p \times p}$ is Hermitian positive-definite~\cite[Theorem 1]{fan1955some}.  
In the case of the Grassmannian manifold, the $QR$ decomposition can be used in place of the polar decomposition, leading to methods with very low computational cost.

Interestingly, in the case of the unitary group and the Grassmannian manifold, superconvergent approximations of the exponential are constructible with this approach.  By this we mean that it is possible to construct polynomials of degree $n$ that, upon projection, deliver approximations to the exponential with error of order $n+2$ or higher.  The appropriate choices of polynomials turn out to be intimately related to the Bessel polynomials, a well-known orthogonal sequence of polynomials~\cite{krall1949new}, and the resulting approximations have error of order $2n+1$; see Theorems~\ref{thm:polarexp} and~\ref{thm:grassmanexp} and Corollary~\ref{thm:grassmanexpqr}.

One of the major advantages of this approach is that it delivers approximations to the exponential on the unitary group, the Grassmannian manifold, and the Steifel manifold that, to machine precision, have orthonormal columns.  This is of obvious importance for the unitary group and the Stiefel manifold, and it is even desirable on the Grassmannian manifold, where it is common in computations to represent elements of the Grassmannian -- $p$-dimensional subspaces of $\mathbb{C}^m$ -- as $m \times p$ matrices whose columns form orthonormal bases for those subspaces~\cite{edelman1998geometry}.

Furthermore, when the polar decomposition is adopted as the projector, projected polynomials have the advantage that they can be computed using only rudimentary operations on matrices: matrix addition, multiplication, and inversion.  This follows from the fact that the polar decomposition can be computed iteratively~\cite[Chapter 8]{higham2008functions}.  Rudimentary algorithms for calculating the exponential on the Grassmannian and Stiefel manifolds are particularly desirable, since the most competitive existing algorithms for accomplishing this task involve singular value and/or eigenvalue decompositions~\cite[Theorem 2.1, Corollary 2.2, and Theorem 2.3]{edelman1998geometry}, a feature that renders existing algorithms less ideally suited for parallel computation than projected polynomials.  

In spite of these advantages, it is worth noting that not all of the constructions in this paper lead to algorithms that outshine their competitors.  
Diagonal Pad\'{e} approximations of the exponential deliver, to machine precision, unitary approximations of $e^\Omega$ when $\Omega$ is skew-Hermitian~\cite[p. 97]{iserles2000lie}.  It is clear that the projected polynomials we present below (in Theorem~\ref{thm:polarexp}) for approximating $e^\Omega$ are more expensive to compute, at least when the comparison is restricted to approximations of $e^\Omega$ with equal order of accuracy.  On the Stiefel manifold, Pad\'{e} approximation is not an option, rendering projected polynomials more attractive.  However, they are not superconvergent on the Stiefel manifold; see Theorem~\ref{thm:stiefelexp}.  The setting in which projected polynomials appear to shine the brightest is the Grassmannian manifold $Gr(p,m)$, where they provide superconvergent, orthonormal approximations to the Riemannian exponential with algorithmic complexity $O(mp^2)$; see Theorem~\ref{thm:grassmanexp} and Corollary~\ref{thm:grassmanexpqr}.  To our knowledge, these are the first such approximations (other than the lowest-order versions) to appear in the literature on the Grassmannian manifold.

Structure-preserving approximations of the exponential map on matrix manifolds have a long history, particularly for matrix manifolds that form Lie groups.  
On Lie groups, techniques involving rational approximation~\cite[p. 97]{iserles2000lie},  splitting~\cite{celledoni2000approximating,yoshida1990construction}, canonical coordinates of the second kind~\cite{celledoni2001methods}, and the generalized polar decomposition~\cite{iserles2005efficient} have been studied, and many of these strategies lead to high-order approximations. 
For more general matrix manifolds like the Grassmannian and Stiefel manifolds, attention has been primarily restricted to methods for calculating the exponential exactly~\cite{edelman1998geometry,absil2004riemannian,absil2009optimization} or approximating it to low order~\cite{absil2009optimization,absil2012projection,kaneko2013empirical,fiori2015tangent}.  In this context, structure-preserving approximations of the exponential are commonly referred to as \emph{retractions}~\cite[Definition 4.1.1]{absil2009optimization}. High-order retractions on the Grassmannian and Stiefel manifolds have received very little attention, but there are good reasons to pursue them.  For instance, in optimization, exactly evaluating the Riemannian Hessian of a function defined on a matrix manifold requires the use of a retraction with second-order accuracy or higher, at least if one is interested in its value away from critical points~\cite[p. 107]{absil2009optimization}.  In addition, existing algorithms for calculating the exponential on the Grassmannian and Stiefel manifolds exactly~\cite[Theorem 2.1, Corollary 2.2, and Theorem 2.3]{edelman1998geometry} are relatively expensive, which raises the question of whether more efficient options, perhaps with nonzero but controllable error, are available. 
On the Grassmannian manifold, the answer seems to be yes, at least for small-normed input; see Corollary~\ref{thm:grassmanexpqr}. 

\paragraph{Organization} 
This paper is organized as follows. In Section~\ref{sec:results}, we give statements of our results, deferring their proofs to Section~\ref{sec:proofs}.  
Our main results are Theorems~\ref{thm:polarexp},~\ref{thm:grassmanexp}, and~\ref{thm:stiefelexp}, which detail families of approximants to the exponential on the unitary group, the Grassmannian manifold, and the Steifel manifold, respectively.  A fourth noteworthy result is Corollary~\ref{thm:grassmanexpqr}, which provides a computationally inexpensive variant of the approximants in Theorem~\ref{thm:grassmanexp}.
We also detail two related results, Proposition~\ref{thm:polargeodesic} and Theorem~\ref{thm:geoarith}, that concern the supercloseness of the geometric and arithmetic means of unitary matrices.  In Section~\ref{sec:proofs}, we prove each of the results just mentioned.  In Section~\ref{sec:numerical}, we describe algorithms for calculating our proposed approximations, with an emphasis on iterative methods for computing the polar decomposition.  We conclude that section with numerical examples.

\section{Statement of Results} \label{sec:results}

In this section, we give statements of our results.  Proofs are detailed in Section~\ref{sec:proofs}.

\subsection{Exponentiation on the Unitary Group} \label{sec:unitaryexp}

Our first result deals with the approximation of the exponential of a skew-Hermitian matrix $\Omega \in \mathbb{C}^{m \times m}$ with projected polynomials.  To motivate the forthcoming theorem, consider the Taylor polynomial $q_n(t\Omega)$ of degree $n$ for $e^{t\Omega}$:
\[
q_n(t\Omega) = \sum_{k=0}^n \frac{(t\Omega)^k}{k!}.
\]
This quantity, in general, is not an unitary matrix, even though the matrix it aims to approximate, $e^{t\Omega}$, is unitary.  If $t$ is sufficiently small (small enough so that $q_n(t\Omega)$ is nonsingular), then $q_n(t\Omega)$ can be made unitary by computing the polar decomposition $q_n(t\Omega) = UH$, where $U \in \mathbb{C}^{m \times m}$ is unitary and $H \in \mathbb{C}^{m \times m}$ is Hermitian positive-definite.  The matrix $U$ is easily seen to be an unitary approximation to $e^{t\Omega}$ with error at worst $O(t^{n+1})$, owing to the fact that
\[
\|U-q_n(t\Omega)\| \le  \|V-q_n(t\Omega)\|
\]
for every unitary matrix $V \in \mathbb{C}^{m \times m}$, 
where $\|\cdot\|$ denotes the Frobenius norm~\cite[Theorem 1]{fan1955some}.  
Indeed,
\begin{align*}
\|U-e^{t\Omega}\| 
&\le \|U-q_n(t\Omega)\| + \|q_n(t\Omega)-e^{t\Omega}\| \\
&\le \|e^{t\Omega}-q_n(t\Omega)\| + \|q_n(t\Omega)-e^{t\Omega}\| \\
&= O(t^{n+1}).
\end{align*}

Below we address the question of whether a better approximation to $e^{t\Omega}$ can be constructed by computing the unitary factor in the polar decomposition of
\[
q_n(t\Omega) = \sum_{k=0}^n a_k t^k\Omega^k
\]
for suitably chosen coefficients $a_k$.  We show that if the coefficients $a_k$ are chosen carefully, then an approximation with error of order $t^{2n+1}$ can be constructed.  The choice of coefficients $a_k$ corresponds to the selection of a Bessel polynomial of degree $n$ in $t\Omega$.   In what follows, we use $\mathcal{P}$ to denote the map which sends a full-rank matrix $A \in \mathbb{C}^{m \times p}$ ($m \ge p$) to the factor $\mathcal{P}(A)=U$ in its polar decomposition $A=UH$, where $U \in \mathbb{C}^{m \times p}$ has orthonormal columns and $H \in \mathbb{C}^{p \times p}$ is Hermitian positive-definite~\cite[Theorem 8.1]{higham2008functions}.

\begin{theorem} \label{thm:polarexp}
Let $\Omega \in \mathbb{C}^{m \times m}$ be skew-Hermitian, and let $n \ge 0$ be an integer.  Define
\begin{equation} \label{bessel}
\Theta_n(z) = \sum_{k=0}^n \binom{n}{k} \frac{(2n-k)!}{(2n)!}(2z)^k.
\end{equation}
Then
\[
\mathcal{P}(\Theta_n(t\Omega)) = e^{t\Omega} + O(t^{2n+1}).
\]
\end{theorem}
In fact, we will show that the polar decomposition of $\Theta_n(t\Omega)$ delivers the highest order approximation of $e^{t\Omega}$ among all polynomials in $t\Omega$ of degree $n$, up to rescaling.  That is, if $r_n(t\Omega)$ is any other polynomial in $t\Omega$ of degree $n$ satisfying $r_n(0)=I$, then
\[
\mathcal{P}(r_n(t\Omega)) = e^{t\Omega} + O(t^k)
\]
for some $0 \le k \le 2n$.

The polynomials~(\ref{bessel}) are scaled versions of Bessel polynomials~\cite{krall1949new}.  More precisely, we have
\[
\Theta_n(z) = \frac{2^n n!}{(2n)!} \theta_n(z),
\]
where
\[
\theta_n(z) = \sum_{k=0}^n \frac{(n+k)!}{(n-k)!k!} \frac{z^{n-k}}{2^k}
\]
denotes the \emph{reverse Bessel polynomial} of degree $n$.  The first few polynomials $\Theta_n(z)$ are given by
\begin{align*}
\Theta_0(z) &= 1, \\
\Theta_1(z) &= 1+z, \\
\Theta_2(z) &= 1+z + \frac{1}{3}z^2, \\
\Theta_3(z) &= 1+z + \frac{2}{5}z^2 + \frac{1}{15} z^3, \\
\Theta_4(z) &= 1+z + \frac{3}{7}z^2 + \frac{2}{21} z^3 + \frac{1}{105} z^4.
\end{align*}
Note that, rather surprisingly, $\Theta_n(z)$ agrees with $e^z$ only to first order for every $n \ge 1$.

\subsection{Exponentiation on the Grassmannian}

We now consider the task of approximating the Riemannian exponential map on the Grassmannian manifold $Gr(p,m)$, which consists of all $p$-dimensional subspaces of $\mathbb{C}^m$, where $m \ge p$.  
We begin by reviewing the geometry of $Gr(p,m)$, with an emphasis on computational aspects.

In computations, it is convenient to represent each subspace $\mathcal{V} \in Gr(p,m)$ with a matrix $Y \in \mathbb{C}^{m \times p}$ having orthonormal columns that span $\mathcal{V}$.  The choice of $Y$ is not unique, so we are of course thinking of $Y$ as a representative of an equivalence class of $m \times p$ matrices sharing the same column space.  With this identification, the tangent space to $Gr(p,m)$ at $Y$ is given by
\[
T_Y Gr(p,m) = \{ Y_\perp K \mid K \in \mathbb{C}^{(m-p) \times p} \},
\]
where $Y_\perp \in \mathbb{C}^{m \times (m-p)}$ is any matrix such that $\begin{pmatrix} Y & Y_\perp \end{pmatrix}$ is unitary, and $Y_\perp K$ is regarded as a representative of an equivalence class of matrices sharing the same column space~\cite[p. 15]{edelman1998geometry}.
With respect to the canonical metric on $Gr(p,m)$, the Riemannian exponential $\mathrm{Exp}^{Gr}_Y : T_Y Gr(p,m) \rightarrow Gr(p,m)$ at $Y \in \mathbb{C}^{m \times p}$ in the direction $H = Y_\perp K \in \mathbb{C}^{m \times p}$ is given by~\cite[p. 10]{edelman1998geometry}
\begin{equation} \label{grassmanexp}
\mathrm{Exp}^{Gr}_Y H = \begin{pmatrix} Y & Y_\perp \end{pmatrix} \exp \begin{pmatrix} 0 & -K^* \\ K & 0 \end{pmatrix} \begin{pmatrix} I \\ 0 \end{pmatrix}.
\end{equation}

The goal of this subsection is to construct computationally inexpensive approximations of $\text{Exp}^{Gr}_Y H$.
In order to be competitive with existing methods, such approximations must have computational complexity $O(mp^2)$ or better, owing to the following well-known result~\cite[Theorem 2.3]{edelman1998geometry}.
\begin{theorem}{\cite[Theorem 2.3]{edelman1998geometry}}
Let $H = U \Sigma V^*$ be the thin singular value decomposition of $H$, i.e. $U \in \mathbb{C}^{m \times p}$ has orthonormal columns, $\Sigma \in \mathbb{C}^{p \times p}$ is diagonal with nonnegative entries, and $V \in \mathbb{C}^{p \times p}$ is unitary.  Then
\[
\mathrm{Exp}^{Gr}_Y H = Y V \cos(\Sigma) V^* + U\sin(\Sigma) V^*.
\]
\end{theorem}

The preceding theorem reveals that $\mathrm{Exp}^{Gr}_Y H$ can be computed exactly with $O(mp^2)$ operations, since this is the cost of computing the thin singular value decomposition of $H \in \mathbb{C}^{m \times p}$.  With this in mind, we aim to derive approximations of $\mathrm{Exp}^{Gr}_Y H$ with smaller or comparable computational complexity.%lower or comparable cost.

Since the matrix $Z := \begin{pmatrix} 0 & -K^* \\ K & 0 \end{pmatrix}$ appearing in~(\ref{grassmanexp}) is skew-Hermitian, an obvious option is to approximate $\exp Z$ in~(\ref{grassmanexp}) with a projected polynomial $\mathcal{P}(\Theta_n(Z))$ in accordance with Section~\ref{sec:unitaryexp}.  This leads to approximants of the form
\begin{equation} \label{grassmanexpnaive}
%\mathrm{Exp}_Y H \approx \begin{pmatrix} Y & Y_\perp \end{pmatrix} \mathcal{P}(\Theta_n(Z)) \begin{pmatrix} I \\ 0 \end{pmatrix},
\mathrm{Exp}^{Gr}_Y (tH) = \begin{pmatrix} Y & Y_\perp \end{pmatrix} \mathcal{P}(\Theta_n(tZ)) \begin{pmatrix} I \\ 0 \end{pmatrix} + O(t^{2n+1}),
\end{equation}
which, unfortunately, have computational complexity $O(m^3)$.  Remarkably, we show in Lemma~\ref{lemma:polarsquare2rect} below that
\begin{equation} \label{PTheta}
\mathcal{P}(\Theta_n(tZ)) \begin{pmatrix} I \\ 0 \end{pmatrix} = \mathcal{P}\left( \Theta_n(tZ) \begin{pmatrix} I \\ 0 \end{pmatrix} \right)
\end{equation}
if $t$ is sufficiently small (small enough so that $\Theta_n(tZ)$ is nonsingular).  This is significant, since the right-hand side of this equality involves the polar decomposition of an $m \times p$ matrix $\Theta_n(tZ) \begin{pmatrix} I \\ 0 \end{pmatrix}$, which can be computed in $O(mp^2)$ operations.  A few more algebraic manipulations (detailed in Section~\ref{sec:grassmanexpproof}) lead to the following scheme for approximating the exponential on the Grassmannian in $O(mp^2)$ operations.
\begin{theorem} \label{thm:grassmanexp}
Let $Y \in \mathbb{C}^{m \times p}$ have orthonormal columns, and let $H \in T_Y Gr(p,m)$.  Then, for any $n \ge 0$,
\[
\mathrm{Exp}^{Gr}_Y (tH) = \mathcal{P}\left( Y\alpha_n(t^2 H^*H) + tH\beta_n(t^2 H^*H) \right) + O\left( t^{2n+1} \right),
\]
where
\begin{align*}
\alpha_n(z) &= \sum_{j=0}^{\floor{n/2}} a_{2j} (-z)^j, \\
\beta_n(z) &= \sum_{j=0}^{\floor{(n-1)/2}} a_{2j+1} (-z)^j,
\end{align*}
and
\[
a_k = \binom{n}{k}\frac{(2n-k)!}{(2n)!} 2^k, \quad k=0,1,\dots,n.
\]

\end{theorem}

The first few nontrivial approximants provided by Theorem~\ref{thm:grassmanexp} read
\begin{align}
\mathrm{Exp}^{Gr}_Y (tH) &= \mathcal{P}\left( Y + tH \right) + O(t^3), \label{grassmanexp1}  \\
\mathrm{Exp}^{Gr}_Y (tH) &= \mathcal{P}\left( Y\left(I-\frac{1}{3}t^2 H^*H\right) + tH \right) + O(t^5), \label{grassmanexp2} \\
\mathrm{Exp}^{Gr}_Y (tH) &= \mathcal{P}\left( Y\left(I-\frac{2}{5}t^2 H^*H\right) + tH\left(I-\frac{1}{15}t^2 H^*H\right) \right) + O(t^7), \label{grassmanexp3} \\
\begin{split}
\mathrm{Exp}^{Gr}_Y (tH) &= \mathcal{P}\left( Y\left(I-\frac{3}{7}t^2 H^*H+\frac{1}{105}t^4 (H^*H)^2\right) \right.\\&\left.\hspace{3em}+ tH\left(I-\frac{2}{21}t^2 H^*H\right) \right) + O(t^9). \label{grassmanexp4} 
\end{split}
\end{align}
Note that, rather interestingly, the commonly used retraction $\mathcal{P}(Y+tH)$ (see, for instance,~\cite{absil2004riemannian}) is in fact an approximation of $\mathrm{Exp}^{Gr}_Y(tH)$ with error $O(t^3) $, despite its appearance.

Note also that the approximants provided by Theorem~\ref{thm:grassmanexp} are rotationally equivariant.  That is, if $V \in \mathbb{C}^{m \times m}$ is a unitary matrix, $\widetilde{Y} = VY$, and $\widetilde{H}=VH$, then $\widetilde{H}^*\widetilde{H}=H^*V^*VH=H^*H$ and hence
\begin{align}
\mathcal{P}\left( \widetilde{Y}\alpha_n(t^2 \widetilde{H}^*\widetilde{H}) + t\widetilde{H}\beta_n(t^2 \widetilde{H}^*\widetilde{H}) \right) &= \mathcal{P}\left( V\left(Y\alpha_n(t^2 H^*H) + tH\beta_n(t^2 H^*H)\right) \right) \nonumber \\
&= V \mathcal{P}\left( Y\alpha_n(t^2 H^*H) + tH\beta_n(t^2 H^*H) \right), \label{grassmaninvariant}
\end{align}
where the last line follows from the fact that 
\begin{equation} \label{polarinvariance}
\mathcal{P}(VA) = V\mathcal{P}(A)
\end{equation}
for every full-rank $A \in \mathbb{C}^{m \times p}$ ($m \ge p$) and every unitary $V \in \mathbb{C}^{m \times m}$.

\paragraph{Replacing the Polar Decomposition with the QR Decomposition}

An extraordinary feature of Theorem~\ref{thm:grassmanexp} is that it applies, with slight modification, even if the map $\mathcal{P}$ is replaced by the map $\mathcal{Q}$ which sends a full-rank matrix $A \in \mathbb{C}^{m \times p}$ ($m \ge p$) to the factor $Q$ in the QR decomposition $A=QR$, where $Q \in \mathbb{C}^{m \times p}$ has orthonormal columns and $R \in \mathbb{C}^{p \times p}$ is upper triangular.  (The map $\mathcal{Q}$ is denoted $\mathrm{qf}$ in~\cite{absil2009optimization}.)  This follows from the fact that the columns of $A$, $\mathcal{P}(A)$, and $\mathcal{Q}(A)$ span the same space, so $\mathcal{P}(A)$ and $\mathcal{Q}(A)$ represent the same element of the Grassmannian manifold.

The only modification of Theorem~\ref{thm:grassmanexp} needed to make this idea precise is to use a genuine distance on the Grassmannian, such as
\[
%\mathrm{dist}^{Gr}(X,Y) = \|XX^*-YY^*\|.
%\mathrm{dist}^{Gr}(X,Y) = \min_{\substack{V_1,V_2 \in \mathbb{C}^{p \times p}\\V_1^*V_1=V_2^*V_2=I}} \|XV_1 - YV_2\|
\mathrm{dist}^{Gr}(X,Y) = \min_{\substack{V,W \in \mathbb{C}^{p \times p}\\V^*V=W^*W=I}} \|XV - YW\|,
\]
to measure the distance between subspaces~\cite[p. 30]{edelman1998geometry}.
We then have the following corollary.

\begin{corollary} \label{thm:grassmanexpqr}
Let $Y \in \mathbb{C}^{m \times p}$ have orthonormal columns, and let $H \in T_Y Gr(p,m)$.  Then, for any $n \ge 0$,
\[
\mathrm{dist}^{Gr} (\mathcal{Q}\left( Y\alpha_n(t^2 H^*H) + tH\beta_n(t^2 H^*H) \right), \, \mathrm{Exp}_Y(tH)) = O(t^{2n+1}),
\]
where $\alpha_n$ and $\beta_n$ are given in the statement of Theorem~\ref{thm:grassmanexp}.
\end{corollary}

This corollary is quite powerful, since the factor $Q$ in the QR decomposition of an $m \times p$ matrix can be computed in merely $2mp^2 - \frac{2}{3}p^3$ operations~\cite[Appendix C]{higham2008functions}, rendering the approximations $\mathcal{Q}\left( Y\alpha_n(t^2 H^*H) + tH\beta_n(t^2 H^*H) \right)$ very cheap to compute.  Note also that these approximations are rotationally equivariant, by an argument similar to the one leading up to~(\ref{grassmaninvariant}).

\subsection{Exponentiation on the Stiefel Manifold} \label{sec:stiefelexp}

The next manifold we consider is the Stiefel manifold
\[
St(p,m) = \{Y \in \mathbb{C}^{m \times p} \mid Y^*Y=I\}.
\]
Unlike the Grassmannian, here we do not regard matrices $Y \in \mathbb{C}^{m \times p}$ as representatives of equivalence classes; each $Y \in \mathbb{C}^{m \times p}$ corresponds to a distinct element of $St(p,m)$.
The tangent space to $St(p,m)$ at $Y \in St(p,m)$ is given by
\[
T_Y St(p,m) = \{Y\Omega + Y_\perp K \mid \Omega = -\Omega^* \in \mathbb{C}^{p \times p}, K \in \mathbb{C}^{(m-p) \times p} \},
\]
where $Y_\perp \in \mathbb{C}^{m \times (m-p)}$ is any matrix such that $\begin{pmatrix} Y & Y_\perp \end{pmatrix}$ is unitary~\cite[Equation 2.5]{edelman1998geometry}.  
With respect to the canonical metric on $St(p,m)$~\cite[Section 2.4]{edelman1998geometry}, the Riemannian exponential $\mathrm{Exp}^{St}_Y : T_Y St(p,m) \rightarrow St(p,m)$ at $Y \in \mathbb{C}^{m \times p}$ in the direction $H = Y\Omega + Y_\perp K \in \mathbb{C}^{m \times p}$ is given by~\cite[Equation 2.42]{edelman1998geometry}
\begin{equation} \label{stiefelexp}
\mathrm{Exp}^{St}_Y H = \begin{pmatrix} Y & Y_\perp \end{pmatrix} \exp \begin{pmatrix} \Omega & -K^* \\ K & 0 \end{pmatrix} \begin{pmatrix} I \\ 0 \end{pmatrix}.
\end{equation}
As an aside, we remark that a different exponential map is obtained if one endows $St(p,m)$ with the metric inherited from the embedding of $St(p,m)$ in Euclidean space~\cite[Section 2.2]{edelman1998geometry}.  We do not consider the latter exponential map in this paper.

There exist algorithms for calculating~(\ref{stiefelexp}) in $O(mp^2)$ operations, the simplest of which involves calculating the $QR$ decomposition of a certain $m \times p$ matrix and then exponentiating a $2p \times 2p$ skew symmetric matrix~\cite[Corollary 2.2]{edelman1998geometry}.  Our aim below is to derive a competitive algorithm for approximating~(\ref{stiefelexp}) to high order using projected polynomials.  

Before doing so, it is important to note that the right-hand side of~(\ref{stiefelexp}) reduces to more familiar expressions in two special cases.  First, when $Y^* H = \Omega = 0$, the right-hand side of~(\ref{stiefelexp}) coincides with the right-hand side of~(\ref{grassmanexp}), the Riemmannian exponential on the Grassmannian.  
On the other hand, if $m=p$ and $Y=I$, then $K$ and $Y_\perp$ are empty matrices and the right-hand side of~(\ref{stiefelexp}) reduces to $e^\Omega$, the exponential of a skew-Hermitian matrix.  Thus, in an effort to generalize Theorems~\ref{thm:polarexp} and~\ref{thm:grassmanexp}, we seek to approximate~(\ref{stiefelexp}) with projected polynomials that reduce to the ones appearing in Theorems~\ref{thm:polarexp} and~\ref{thm:grassmanexp} in those special cases.

In view of Theorem~\ref{thm:grassmanexp} and the identities $Y^*H=\Omega$ and $H^*H = -\Omega^2+K^*K$, it is natural to consider approximations of~(\ref{stiefelexp}) of the form
\begin{equation} \label{stiefelansatz}
\mathrm{Exp}^{St}_Y (tH) \approx \mathcal{P}(Yq(t^2 H^*H,tY^*H) + tHr(t^2H^*H,tY^*H)),
\end{equation}
where $q(x,y)$ and $r(x,y)$ are polynomials in the (non-commuting) variables $x$ and $y$.  In order to ensure that these approximations recover those appearing in Theorems~\ref{thm:polarexp} and~\ref{thm:grassmanexp}, we insist that:
\begin{enumerate}[label=(\ref*{sec:stiefelexp}.\roman*),ref=\ref*{sec:stiefelexp}.\roman*]
\item \label{polycond1} $q(x,0)=\alpha_n(x)$.
\item \label{polycond2} $r(x,0)=\beta_n(x)$. 
\item \label{polycond3} $q(-x^2,x)$ and $r(-x^2,x)$ are polynomials of degree at most $n$ and $n-1$, respectively, satisfying
\[
q(-x^2,x) + x r(-x^2,x) = \Theta_n(x).
\]
\end{enumerate}

It turns out that such approximations can be constructed, but they lack the superconvergence enjoyed by the approximations in Theorems~\ref{thm:polarexp} and~\ref{thm:grassmanexp} (unless $\Omega=0$ or $m=p$).  The difficulty becomes apparent if one compares $\mathcal{P}(Y+tH)$ with $\mathrm{Exp}^{St}_Y(tH)$ for generic $Y \in St(p,m)$ and $H = Y\Omega + Y_\perp K \in T_Y St(p,m)$.  As $t \rightarrow 0$, one observes numerically that $\mathcal{P}(Y+tH) = \mathrm{Exp}^{St}_Y(tH) + O(t^2)$ (unless $\Omega=0$ or $m=p$).\footnote{The astute reader may notice that this appears to contradict Theorem 4.9 of~\cite{absil2012projection}, which states, among other things, that projective retractions (see~\cite[Example 4.5]{absil2012projection}) are automatically second-order.  However, it is not the exponential map~(\ref{stiefelexp}) that $\mathcal{P}(Y+tH)$ approximates to second order.  Rather, it is exponential map associated with the metric inherited from the embedding of $St(p,m)$ in $\mathbb{C}^{m \times p}$.}   This contrasts starkly with the situation in Theorems~\ref{thm:polarexp} and~\ref{thm:grassmanexp}, where the polar decomposition of the first-order Taylor approximant of the exponential had superconvergent error $O(t^3)$.  The following theorem confirms this observation and provides a couple of higher-order approximations of~(\ref{stiefelexp}).  In it, we use $\|\cdot\|$ to denote the Frobenius norm.  

\begin{theorem} \label{thm:stiefelexp}
Let $Y \in St(p,m)$ and $H \in T_Y St(p,m)$.  Define
\begin{align*}
\gamma_1(x,y) &= 1, & \delta_1(x,y) &= 1, \\
\gamma_2(x,y) &= 1-\frac{1}{3}x-\frac{1}{2}y^2, & \delta_2(x,y) &= 1+\frac{1}{2}y, \\
\gamma_3(x,y) &= 1-\frac{2}{5}x-\frac{1}{2}y^2-\frac{1}{6}y^3-\frac{1}{6}xy, & \delta_3(x,y) &= 1-\frac{1}{15}x+\frac{1}{2}y.
\end{align*}
For $n=1,2,3$, we have
\begin{equation} \label{stiefelapprox}
\mathrm{Exp}^{St}_Y (tH) = \mathcal{P}\left(Y\gamma_n(t^2 H^*H, tY^*H) + tH\delta_n(t^2 H^*H, tY^*H)\right) + E,
\end{equation}
where
\begin{equation} \label{stiefelapprox2}
E = \begin{cases}
O(t^{2n+1}) &\mbox{ if $Y^*H=0$ or $m=p$}, \\
O(t^{n+1}) &\mbox{ otherwise.}
\end{cases}
\end{equation}
In addition, for every polynomial $q(x,y)$ and $r(x,y)$ satisfying~(\ref{polycond1}-\ref{polycond3}) ($1 \le n \le 3$), there exists $Y \in St(p,m)$, $H \in T_Y St(p,m)$, $C>0$, and $t_0>0$ such that
\begin{equation} \label{stiefelerror}
\left\|\mathrm{Exp}^{St}_Y (tH) - \mathcal{P}\left(Yq(t^2 H^*H, tY^*H) + tHr(t^2 H^*H, tY^*H)\right)\right\| \ge Ct^{n+1}
\end{equation}
for every $t \le t_0$.
\end{theorem}

Written more explicitly, the approximants provided by Theorem~\ref{thm:stiefelexp} read
\begin{align}
\mathrm{Exp}^{St}_Y (tH) &\approx \mathcal{P}(Y + tH), \label{stiefelexp1} \\
\mathrm{Exp}^{St}_Y (tH) &\approx \mathcal{P}\left(Y\left(I-\frac{1}{3}t^2H^*H-\frac{1}{2}t^2(Y^*H)^2\right) + tH\left(I+\frac{1}{2}t Y^*H\right) \right), \label{stiefelexp2} \\
\mathrm{Exp}^{St}_Y (tH) &\approx \mathcal{P}\left(Y \left( I + t^2\left(-\frac{2}{5}H^*H-\frac{1}{2}(Y^*H)^2\right) + t^3\left(-\frac{1}{6}H^*HY^*H-\frac{1}{6}(Y^*H)^3\right)\right) \right.\nonumber\\&\hspace{0.4in}\left. + tH \left(I + \frac{1}{2}tY^*H -\frac{1}{15}t^2H^*H\right) \right). \label{stiefelexp3}
\end{align}
All of these are rotationally equivariant by an argument similar to the one leading up to~(\ref{grassmaninvariant}).  

Observe that when $Y^*H=\Omega=0$, the right-hand sides of~(\ref{stiefelexp1}-\ref{stiefelexp3}) reduce to those in~(\ref{grassmanexp2}-\ref{grassmanexp3}), respectively.  Likewise, when $m=p$ (so that $H=Y\Omega$ and $YY^*=I$), they reduce to 
$Y\mathcal{P}(\Theta_1(tH))$, $Y\mathcal{P}(\Theta_2(tH))$, and $Y\mathcal{P}(\Theta_3(tH))$, 
respectively, which are precisely the approximations of $Ye^{\Omega}$ provided by Theorem~\ref{thm:polarexp}.  See Section~\ref{sec:stiefelexpproof} for details.

In view of the complexity of~(\ref{stiefelexp2}) and~(\ref{stiefelexp3}), we do not believe that a general formula (valid for all $n$) can be conveniently written down for polynomials $q(x,y)$ and $r(x,y)$ satisfying~(\ref{polycond1}-\ref{polycond3}) that deliver approximations of~(\ref{stiefelexp}) of the form~(\ref{stiefelansatz}) with error of optimal order.  (As a matter of fact, the polynomials $\gamma_3(x,y)$ and $\delta_3(x,y)$ are not even uniquely determined by these conditions.)  However, the proofs presented in Section~\ref{sec:stiefelexpproof} demonstrate how one can construct such polynomials.

Note that if the conditions~(\ref{polycond1}-\ref{polycond3}) are relaxed, then it is straightforward to construct approximations of~(\ref{stiefelexp}) with error $O(t^{n+1})$: Simply truncate the Taylor series for $\exp\begin{pmatrix} \Omega & -K^* \\ K & 0 \end{pmatrix}$, insert the result into~(\ref{stiefelexp}), and express the result in terms of $Y$ and $H$ using the identities $H=Y\Omega+Y_\perp K$, $Y^*H=\Omega$, and $H^*H = -\Omega^2+K^*K$.  If desired, the result can be orthonormalized with the map $\mathcal{P}$, retaining the order of accuracy of the approximation. For this reason, Theorem~\ref{thm:stiefelexp} is less powerful than Theorems~\ref{thm:polarexp} and~\ref{thm:grassmanexp}, and it underscores the complexity of the Stiefel manifold relative to the Grassmannian and the unitary group. 
For more evidence of the computational difficulties inherent to the Stiefel manifold, we refer the reader to~\cite{edelman1998geometry}.

\subsection{Geometric and Arithmetic Means of Unitary Matrices}

We conclude this section by stating two results that concern the supercloseness of certain means of unitary matrices.  At the surface, these results might not appear to be closely related to Theorems~\ref{thm:polarexp},~\ref{thm:grassmanexp},~\ref{thm:stiefelexp}, but in fact they follow from the same general theory.

Our first result reveals that the polar decomposition of the (componentwise) linear interpolant of two unitary matrices $U_1$ and $U_2$ is superclose to the geodesic joining $U_1$ and $U_2$.

\begin{proposition} \label{thm:polargeodesic}
Let $U_1 \in \mathbb{C}^{m \times m}$ be unitary, let $\Omega \in \mathbb{C}^{m \times m}$ be skew-Hermitian, and let $U_2 = U_1 e^{t\Omega}$.  Assume that $(1-s)U_1 + s U_2$ is nonsingular for each $s \in [0,1]$. Then
\[
\mathcal{P}((1-s)U_1 + s U_2) = U_1 e^{st\Omega} + O(t^3)
\]
for every $s \in [0,1/2) \cup (1/2,1]$.  When $s=1/2$, the equality $\mathcal{P}((1-s)U_1 + s U_2) = U_1 e^{st\Omega}$ holds exactly.
\end{proposition}

The case $s=1/2$ in the preceding lemma recovers the well-known observation (see~\cite[Theorem 4.7]{higham2005functions} and~\cite[Equation (3.14)]{moakher2002means}) 
%~\cite[Theorem 4.7, Equation (3.14)]{higham2005functions,moakher2002means} 
%~\cites[Theorem 4.7]{higham2005functions}[Equation (3.14)]{moakher2002means} 
that the unitary factor $U=\mathcal{P}(\frac{1}{2}(U_1+U_2))$ in the polar decomposition of $\frac{1}{2}(U_1+U_2)$ is given by
\[
U = U_1 e^{\frac{1}{2}\log(U_1^* U_2)} = U_1 (U_1^* U_2)^{1/2}
\]
whenever $U_1$ and $U_2$ are (non-antipodal) members of the unitary group.

Our second result of this subsection generalizes the preceding proposition in the following way.  Let
\[
\mathbb{A}(U_1,\dots,U_n;w) = \argmin_{\substack{V \in \mathbb{C}^{m \times m},\\ V^*V=I}} \sum_{i=1}^n w_i \|V-U_i\|^2
\]
denote the weighted \emph{arithmetic mean}~\cite[Definition 5.1]{moakher2002means} of unitary matrices $U_1,\dots,U_n \in \mathbb{C}^{m \times m}$, where $w \in \mathbb{R}^n$ is a vector of weights summing to 1.  Let 
\[
\mathbb{G}(U_1,\dots,U_n;w) = \argmin_{\substack{V \in \mathbb{C}^{m \times m}, \\ V^*V=I}} \sum_{i=1}^n w_i \, \mathrm{dist}(V,U_i)^2
\]
denote their weighted \emph{geometric mean}~\cite[Definition 5.2]{moakher2002means}, where
\[
\mathrm{dist}(U,V) = \frac{1}{\sqrt{2}}\|\log(U^*V)\|
\]
denotes the geodesic distance on the unitary group~\cite[Equation (2.6)]{moakher2002means}.  It can be shown~\cite[Proposition 4]{gawlik2016embedding} that $\mathbb{A}(U_1,\dots,U_n;w)$ exists whenever $\sum_{i=1}^n w_i U_i$ is nonsingular, and is given explicitly by
\begin{equation} \label{arithformula}
\mathbb{A}(U_1,\dots,U_n;w) = \mathcal{P}\left(\sum_{i=1}^n w_i U_i \right).
\end{equation}
On the other hand, $\mathbb{G}(U_1,\dots,U_n;w)$ is characterized implicitly by the condition~\cite[p. 14]{moakher2002means}
\begin{equation} \label{geomeanEL}
\sum_{i=1}^n w_i \log( \mathbb{G}(U_1,\dots,U_n;w)^* U_i) = 0.
\end{equation}
The following theorem reveals that if the data $U_1,\dots,U_n$ are nearby, then their weighted arithmetic and geometric means are superclose.

\begin{theorem} \label{thm:geoarith}
Let $U_i : [0,T] \rightarrow \mathbb{C}^{m \times m}$, $i=1,2,\dots,n$, be continuous functions on $[0,T]$ such that $U_i(t)$ is unitary for each $t \in [0,T]$. 
Suppose that there exists $C>0$ such that
\begin{equation} \label{distanceassumption}
%\|U_i(t)-U_j(t)\| \le Ct
\mathrm{dist}(U_i(t),U_j(t)) \le Ct
\end{equation}
for every $i,j=1,2,\dots,n$ and every $t \in [0,T]$.
Then, for any $w \in \mathbb{R}^n$ with entries summing to 1,
\[
\mathbb{A}(U_1(t),\dots,U_n(t);w) = \mathbb{G}(U_1(t),\dots,U_n(t);w) + O(t^3).
\]
\end{theorem}

\section{Proofs} \label{sec:proofs}

In this section, we prove Theorems~\ref{thm:polarexp},~\ref{thm:grassmanexp},~\ref{thm:stiefelexp}, and~\ref{thm:geoarith} and Proposition~\ref{thm:polargeodesic}.  Our proofs are structured as follows.  In Section~\ref{sec:polarpert}, we consider the general problem of estimating
\begin{equation} \label{polardiff}
\|\mathcal{P}(A)-\widetilde{U}\|,
\end{equation}
where $A \in \mathbb{C}^{m \times p}$ ($m \ge p$) is a full-rank matrix and $\widetilde{U} \in \mathbb{C}^{m \times p}$ has orthonormal columns.  We show in Lemma~\ref{lemma:symmetry} that this quantity can be estimated by measuring (1) the extent to which $\widetilde{U}^*A$ fails to be Hermitian and (2) the discrepancy between the range of $A$ and the range of $\widetilde{U}$.  We then leverage this lemma to prove Theorem~\ref{thm:polarexp} in Section~\ref{sec:polarexpproof}, Theorem~\ref{thm:grassmanexp} in Section~\ref{sec:grassmanexpproof}, Theorem~\ref{thm:stiefelexp} in Section~\ref{sec:stiefelexpproof}, and Theorem~\ref{thm:polargeodesic} and Proposition~\ref{thm:geoarith} in Section~\ref{sec:geoarithproof}. 

It turns out that one of the theorems proved below, Theorem~\ref{thm:polarexp}, admits an alternative proof that does not rely on Lemma~\ref{lemma:symmetry}.  This alternative proof, which relies instead on a relationship between projected polynomials and Pad\'{e} approximation, is shorter than the one we present in Section~\ref{sec:polarexpproof}, so we detail it in Section~\ref{sec:pade} for completeness.  
We have chosen to retain both proofs in this paper for several reasons.  The proof in Section~\ref{sec:polarexpproof}, despite being longer, highlights the versatility of Lemma~\ref{lemma:symmetry}, a lemma whose wide-ranging applicability is, in our opinion, one of the key contributions of this paper.
The proof in Section~\ref{sec:polarexpproof} is also more elementary, in a certain sense, than that in Section~\ref{sec:pade}, since the former relies merely on well-known perturbation estimates for the polar decomposition, whereas the latter relies on Pad\'{e} approximation theory and certain results concerning the commutativity of functions of matrices.

\subsection{Perturbations of the Polar Decomposition} \label{sec:polarpert}

We begin our examination of~(\ref{polardiff}) by studying the sensitivity of the polar decomposition to perturbations.
In what follows, we continue to use $\|\cdot\|$ to denote the Frobenius norm. We denote the $i^{th}$ largest singular value of a matrix $A \in \mathbb{C}^{m \times p}$ by $\sigma_i(A)$.  If $m=p$, we use $\rho(A)$ to denote the spectral radius of $A$.  If furthermore $A$ has real eigenvalues, we denote them by $\lambda_1(A) \ge \lambda_2(A) \ge \dots \lambda_p(A)$. Note that with this convention, it need not be true that $|\lambda_1(A)| \ge |\lambda_2(A)| \ge \dots |\lambda_p(A)|$.

We denote by 
\[
\mathrm{sym}(A) = \frac{1}{2}(A+A^*)
\]
and
\[
\mathrm{skew}(A) = \frac{1}{2}(A-A^*)
\]
the Hermitian and skew-Hermitian parts of a square matrix $A$, respectively.  Note that since 
\begin{equation} \label{normtranspose}
\|A\|=\|A^*\|,
\end{equation}
we have
\begin{equation}
\|\mathrm{sym}(A)\| \le \|A\|
\end{equation}
and
\begin{equation} \label{skewbound}
\|\mathrm{skew}(A)\| \le \|A\|
\end{equation}
for any square matrix $A$. 

We will make use of the following additional properties of the Frobenius norm.
For any $A \in \mathbb{C}^{m \times p}$, any $B \in \mathbb{C}^{p \times q}$, any $C \in \mathbb{C}^{p \times p}$, and any $U \in \mathbb{C}^{m \times p}$ with orthonormal columns:
\begin{enumerate}[label=(\ref*{sec:polarpert}.\roman*),ref=\ref*{sec:polarpert}.\roman*]
\item \label{fro1} $\|A^* B\| \le \|A^*\| \sigma_1(B) = \|A\| \sigma_1(B)$~\cite[Equation (B.7)]{higham2008functions}.
\item \label{fro2} $\|U^* A\| \le \|A\|$ (This follows from~(\ref{fro1}) and~(\ref{normtranspose})).
\item \label{fro3} $\|UC\| = \|C\|$~\cite[Problem B.7]{higham2008functions}.
\item \label{fro4}  $\rho(C) \le \|C\|$~\cite[Equation (B.8)]{higham2008functions}.
\end{enumerate}
Note that~(\ref{fro1}) is sharper than the estimate $\|A^* B\| \le \|A^*\| \|B\|$, so we will frequently use~(\ref{fro1}) instead of the latter estimate.

We first recall a result concerning the stability of the polar decomposition under perturbations.  A proof is given in~\cite{li2002perturbation}.
\begin{lemma}{\cite[Theorem 2.4]{li2002perturbation}} \label{lemma:perturbation}
Let $A,\widetilde{A} \in \mathbb{C}^{m \times p}$ ($m \ge p$) be full-rank matrices with polar decompositions $A = UH$ and $\widetilde{A} = \widetilde{U} \widetilde{H}$, where $U,\widetilde{U} \in \mathbb{C}^{m \times p}$ have orthonormal columns and $H,\widetilde{H} \in \mathbb{C}^{p \times p}$ are Hermitian positive-definite.  Then%, in any unitarily invariant norm,
\[
\|U - \widetilde{U}\| \le \frac{2}{\sigma_p(A) + \sigma_p(\widetilde{A})} \|A-\widetilde{A}\|.
\]
\end{lemma}

Next, we consider a full-rank matrix $A \in \mathbb{C}^{m \times p}$ with polar decomposition $A=UH$, and we use the preceding lemma to show that the distance from $U$ to any other matrix $\widetilde{U} \in \mathbb{C}^{m \times p}$ (sufficiently close to $A$) with orthonormal columns is controlled by two properties: (1) the extent to which $\widetilde{U}^* A$ fails to be Hermitian, and (2) the discrepancy between the range of $A$ and the range of $\widetilde{U}$.

\begin{lemma} \label{lemma:symmetry}
Let $A \in \mathbb{C}^{m \times p}$ ($m \ge p$) be a full-rank matrix with polar decomposition $A=UH$, where $U = \mathcal{P}(A) \in \mathbb{C}^{m \times p}$ has orthonormal columns and $H \in \mathbb{C}^{p \times p}$ is Hermitian positive-definite.  Then, for any matrix $\widetilde{U} \in \mathbb{C}^{m \times p}$ with orthonormal columns satisfying $\|A-\widetilde{U}\| < 1$, we have
\begin{equation} \label{ineq}
\frac{\max\left\{ 2\|\mathrm{skew}(\widetilde{U}^*A)\|,  \|(I-\widetilde{U}\widetilde{U}^*)A\| \right\}}{2\sigma_1(A)} \le \|U - \widetilde{U}\|  \le \frac{2\left(\|\mathrm{skew}(\widetilde{U}^* A)\| + \|(I-\widetilde{U}\widetilde{U}^*)A\|\right)}{\sigma_p(A)+\sigma_p(\mathrm{sym}(\widetilde{U}^* A))}.
\end{equation}
\end{lemma}
\begin{proof}
Define $\widetilde{H} = \mathrm{sym}(\widetilde{U}^* A)$.  This matrix is positive-definite, since it is a small-normed perturbation of the identity matrix.  Indeed, the relation
\[
\widetilde{H}  - I = \mathrm{sym}\left(\widetilde{U}^* (A-\widetilde{U})\right)
\]
implies $\|\widetilde{H}-I\| \le \|\widetilde{U}^*(A-\widetilde{U})\|$.  Thus, using~(\ref{fro2}) and~(\ref{fro4}), the smallest eigenvalue of $\widetilde{H}$ satisfies 
\begin{align*}
\lambda_p(\widetilde{H}) 
&= 1 + \lambda_p(\widetilde{H}-I) \\
&\ge 1 - \| \widetilde{H} - I \| \\
&\ge 1 - \|\widetilde{U}^* (A-\widetilde{U})\| \\
&\ge 1 -  \|A-\widetilde{U}\| \\
&> 0.
\end{align*}
Now define $\widetilde{A} = \widetilde{U} \widetilde{H}$.  Observe that
\begin{align*}
A-\widetilde{A} 
&= (\widetilde{U} \widetilde{U}^* + I - \widetilde{U}\widetilde{U}^*) A - \widetilde{U} \mathrm{sym}(\widetilde{U}^* A) \\
&= \widetilde{U} \mathrm{skew}(\widetilde{U}^* A) + (I-\widetilde{U}\widetilde{U}^*) A,
\end{align*}
so
\begin{align*}
\|A-\widetilde{A}\| 
&\le \|\widetilde{U} \mathrm{skew}(\widetilde{U}^* A)\| + \|(I-\widetilde{U}\widetilde{U}^*) A\| \\
&= \|\mathrm{skew}(\widetilde{U}^* A)\| + \|(I-\widetilde{U}\widetilde{U}^*) A\|
\end{align*}
by~(\ref{fro3}).  The right-hand inequality in~(\ref{ineq}) then follows from Lemma~\ref{lemma:perturbation} upon noting that $\widetilde{A}$ and $\mathrm{sym}(\widetilde{U}^* A)$ have the same singular values.

To prove the left-hand inequality in~(\ref{ineq}), observe that since $H=U^* A$ is Hermitian,
\[
\mathrm{skew}(\widetilde{U}^* A) = \mathrm{skew}\left( (\widetilde{U}^*-U^*) A \right).
\]
Thus, using~(\ref{fro1}) and~(\ref{skewbound}),
\begin{align}
\|\mathrm{skew}(\widetilde{U}^* A)\| 
&\le \|(\widetilde{U}^*-U^*) A\| \nonumber \\
&\le \|\widetilde{U}^*-U^*\| \sigma_1(A) \nonumber \\
&= \|\widetilde{U}-U\| \sigma_1(A). \label{bound1}
\end{align}
On the other hand, since $UU^*A=UH=A$, we have
\begin{align*}
(I-\widetilde{U}\widetilde{U}^*)A 
&= (UU^*-\widetilde{U}\widetilde{U}^*)A \\
&= (U-\widetilde{U})U^*A + \widetilde{U}(U-\widetilde{U})^*A.
\end{align*}
Thus, using~(\ref{fro1}),~(\ref{fro3}), and the fact that $\sigma_1(U^*A) = \sigma_1(A)$, it follows that
\begin{align}
\|(I-\widetilde{U}\widetilde{U}^*)A \| 
&\le \|U-\widetilde{U}\| \sigma_1(U^*A) + \|(U-\widetilde{U})^* A\| \nonumber \\
&\le \|U-\widetilde{U}\| \sigma_1(A) + \|U-\widetilde{U}\|\sigma_1(A) \nonumber \\
&= 2\|U-\widetilde{U}\|\sigma_1(A).  \label{bound2}
\end{align}
Combining~(\ref{bound1}) and~(\ref{bound2}) proves the left-hand inequality in~(\ref{ineq}).
\end{proof}

The following less sharp version of Lemma~\ref{lemma:symmetry}, applicable in the square case ($m=p$), will be useful in the upcoming sections.

\begin{lemma} \label{lemma:symmetrysimple}
Let $A$, $U$, and $\widetilde{U}$ be as in Lemma~\ref{lemma:symmetry}.  If $A$ is square (i.e. $m=p$), then
\[
\|A\|^{-1} \|\mathrm{skew}(\widetilde{U}^* A) \| \le \|U-\widetilde{U}\| \le 2\|A^{-1}\| \|\mathrm{skew}(\widetilde{U}^* A) \|.
\]
\end{lemma}
\begin{proof}
Use~(\ref{ineq}) together with the fact that $\sigma_1(A) \le \|A\|$, $\sigma_p(A)^{-1} \le \|A^{-1}\|$, and $\widetilde{U}\widetilde{U}^*=I$ when $\widetilde{U}$ is square.
\end{proof}

\subsection{Exponentiation on the Unitary Group} \label{sec:polarexpproof}

We now prove Theorem~\ref{thm:polarexp}.  
Fix an integer $n \ge 0$ and consider a polynomial of the form
\[
q_n(z) = \sum_{k=0}^n a_k z^k,
\]
with $a_0=1$ and $a_k \in \mathbb{C}$, $k=1,2,\dots,n$.  We aim to find coefficients $a_k$ making $\mathcal{P}(q_n(t\Omega))-e^{t\Omega}$ small for any skew-Hermitian matrix $\Omega$.  Applying Lemma~\ref{lemma:symmetrysimple} with $A=q_n(t\Omega)$ and $\widetilde{U} = e^{t\Omega}$ gives
\begin{equation} \label{polarexpineq}
\|q_n(t\Omega)\|^{-1} \|\mathrm{skew}(e^{-t\Omega} q_n(t\Omega))\| \le \|\mathcal{P}(q_n(t\Omega)) - e^{t\Omega}\| \le 2\|q_n(t\Omega)^{-1}\| \|\mathrm{skew}(e^{-t\Omega} q_n(t\Omega))\|, 
\end{equation}
provided that $t$ is sufficiently small (small enough that $q_n(t\Omega)$ has full rank and $\|q_n(t\Omega)-e^{t\Omega}\|<1$).

This inequality is of great utility, since $\|q_n(t\Omega)\|^{-1}$ and $\|q_n(t\Omega)^{-1}\|$ are each $O(1)$ as $t \rightarrow 0$, and
\[
\mathrm{skew}(e^{-t\Omega}q_n(t\Omega)) = \frac{1}{2} \left( e^{-t\Omega} q_n(t\Omega) - q_n(-t\Omega) e^{t\Omega} \right)
\]
can be expanded in powers of $t$.
Namely,
%\begin{align*}
%e^{-t\Omega} q_n(t\Omega) - q_n(-t\Omega) e^{t\Omega} 
%&= \left( \sum_{j=0}^\infty (-1)^j \frac{(t\Omega)^j}{j!} \right) \left( \sum_{k=0}^n a_k (t\Omega)^k \right) \\ 
%&\hspace{1em} - \left( \sum_{k=0}^n (-1)^k a_k (t\Omega)^k \right) \left( \sum_{j=0}^\infty \frac{(t\Omega)^j}{j!} \right)  \\
%&= \sum_{l=0}^\infty b_l t^l \Omega^l,
%\end{align*}
\begin{align*}
e^{-t\Omega} q_n(t\Omega) - q_n(-t\Omega) e^{t\Omega} 
&= \sum_{j=0}^\infty (-1)^j \frac{(t\Omega)^j}{j!} \sum_{k=0}^n a_k (t\Omega)^k - \sum_{k=0}^n (-1)^k a_k (t\Omega)^k \sum_{j=0}^\infty \frac{(t\Omega)^j}{j!}  \\
&= \sum_{l=0}^\infty b_l t^l \Omega^l,
\end{align*}
where
\begin{align*}
b_l 
&= \sum_{k=0}^{\mathrm{min}(l,n)} \frac{1}{(l-k)!} \left( (-1)^{l-k} + (-1)^{k+1} \right) a_k \\
&= \begin{cases}
\sum_{k=0}^{\mathrm{min}(l,n)} \frac{2(-1)^{k+1}}{(l-k)!} a_k, &l \text{ odd}, \\
0, &l \text{ even}.
\end{cases}
\end{align*}
The quantity $e^{-t\Omega} q_n(t\Omega) - q_n(-t\Omega) e^{t\Omega}$ is thus of the highest order in $t$ when the $n$ coefficients $a_k$, $k=1,2,\dots,n$, are chosen to make $b_l=0$ for $l=1,3,5,\dots,2n-1$.  This is achieved when
\begin{equation} \label{ak}
a_k = \binom{n}{k}\frac{(2n-k)!}{(2n)!} 2^k, \quad k=1,2,\dots,n,
\end{equation}
as the following lemma shows.
\begin{lemma}
With $a_0=1$ and $a_k$ given by~(\ref{ak}) for $1 \le k \le n$, we have
\begin{equation} \label{linearsystem}
\sum_{k=0}^{\mathrm{min}(l,n)} \frac{2(-1)^{k+1}}{(l-k)!} a_k = 0, \quad l=1,3,5,\dots,2n-1.
\end{equation}
\end{lemma}
\begin{proof}
Substitution gives
\begin{align*}
\sum_{k=0}^{\mathrm{min}(l,n)} \frac{2(-1)^{k+1}}{(l-k)!} a_k 
&= \sum_{k=0}^{\mathrm{min}(l,n)} (-2)^{k+1} \binom{n}{k}\frac{(2n-k)!}{(2n)!(l-k)!} \\
&= \frac{-2(n!)^2}{l!(2n)!} \sum_{k=0}^{\mathrm{min}(l,n)} (-2)^k \binom{l}{k}\binom{2n-k}{n},
\end{align*}
so it suffices to show that
\[
\sum_{k=0}^{\mathrm{min}(l,n)} (-2)^k \binom{l}{k}\binom{2n-k}{n} = 0
\]
for each $l=1,3,5,\dots,2n-1$.  To prove this, consider the polynomial
\[
r(z) =  \sum_{k=0}^{2n} r_k z^k = (1+z)^{2n-l} (z-1)^{l}.
\]
The coefficient of $z^n$ in this polynomial is precisely
\begin{equation} \label{pn}
r_n = \sum_{k=0}^{\mathrm{min}(l,n)} (-2)^k \binom{l}{k}\binom{2n-k}{n}.
\end{equation}
Indeed,
\begin{align*}
r(z) &= (1+z)^{2n-l} (z-1)^{l} \\
&= (1+z)^{2n-2l} \left(-2(1+z) + (1+z)^2 \right)^l \\
&= (1+z)^{2n-2l} \sum_{k=0}^l \binom{l}{k} (-2)^k (1+z)^k (1+z)^{2(l-k)} \\ 
&= \sum_{k=0}^l \binom{l}{k} (-2)^k (1+z)^{2n-k} \\
&= \sum_{k=0}^l \binom{l}{k} (-2)^k \sum_{j=0}^{2n-k} \binom{2n-k}{j} z^j,
\end{align*}
and taking $j=n$ in the inner summation above gives~(\ref{pn}).  Now observe that $r(z)$ satisfies the symmetry
\[
r(z) = (-1)^l z^{2n} r(z^{-1}).
\]
From this it follows that the coefficients $r_k$ satisfy $r_k=(-1)^l r_{2n-k}$, $k=0,1,\dots,2n$.  In particular, $r_n = (-1)^l r_n$, so $r_n=0$ when $l$ is odd.
\end{proof}

This completes the proof of Theorem~\ref{thm:polarexp}, since by the inequality~(\ref{polarexpineq}), the unitary factor in the polar decomposition of the polynomial $\sum_{k=0}^n a_k t^k \Omega^k$ with coefficients given by~(\ref{ak}) delivers an approximation of $e^{t\Omega}$ with error of order $t^{2n+1}$.  

Uniqueness of the solution~(\ref{ak}) to~(\ref{linearsystem}) is a consequence of the following lemma, which proves that the linear system~(\ref{linearsystem}) is nonsingular.

\begin{lemma}
Fixing $a_0 = 1$, the linear system~(\ref{linearsystem}) in the $n$ unknowns $a_1,a_2,\dots,a_n$ is nonsingular.
\end{lemma}
\begin{proof}
Upon rearrangement,~(\ref{linearsystem}) reads
\[
M x = y,
\]
where $x,y \in \mathbb{C}^n$ and $M \in \mathbb{C}^{n \times n}$ have entries given by
\begin{align*}
x_i &= a_i, \quad i=1,2,\dots,n, \\
y_i &= \frac{2}{(2i-1)!}, \quad i=1,2,\dots,n, \\
M_{ij} &= 
\begin{cases}
\frac{2(-1)^{j+1}}{(2i-1-j)!}, &\mbox{ if } j \le \min(2i-1,n), \\
0, &\mbox{ otherwise. }
\end{cases}
\end{align*}
An inductive argument shows that the determinant of $M$ is equal to
\[
\det M = \frac{(-1)^{\floor{\frac{n}{2}}}(-2)^n}{\prod_{k=1}^{n-1} (2k-1)!!},
\]
where $l!! = \prod_{j=0}^{\ceil{l/2}-1} (l-2j)$ denotes the double factorial.  In particular, $\det M \neq 0$, showing that $M$ is nonsingular.
\end{proof}

\subsection{Connections with Pad\'{e} Approximation} \label{sec:pade}

We now present an alternative proof of Theorem~\ref{thm:polarexp} that relies not on Lemma~\ref{lemma:symmetry}, but rather on a connection between $\mathcal{P}(\Theta_n(t\Omega))$ and the diagonal Pad\'{e} approximant of $e^{2t\Omega}$.  

Our alternative proof will make use of the fact that if $A \in \mathbb{C}^{m \times p}$ ($m \ge p$) has full rank, then
\begin{equation} \label{polarformula}
\mathcal{P}(A) = A (A^*A)^{-1/2},
\end{equation}
where $C^{-1/2}$ denotes inverse of the principal square root of a square matrix $C$ with no nonpositive real eigenvalues~\cite[Theorem 8.1]{higham2008functions}.

It will also make use of the following facts: If $f$ and $g$ are two scalar-valued functions defined on the spectrum of a square matrix $A$, then $f(A)$ and $g(A)$ are well-defined~\cite[Section 1.2]{higham2008functions}, $f(A)$ commutes with $g(A)$~\cite[Theorem 1.13(e)]{higham2008functions}, and the spectrum of $f(A)$ is the image of the spectrum of $A$ under $f$~\cite[Theorem 1.13(d)]{higham2008functions}.

\begin{lemma}
Let $q_n(z)$ be a polynomial of degree $n \ge 0$ with $q_n(0) \neq 0$, and let $\Omega \in \mathbb{C}^{m \times m}$ be skew-Hermitian.  For each $t$ sufficiently small (small enough so that $q_n(t\Omega)$ is nonsingular), we have
\[
\mathcal{P}(q_n(t\Omega))^2 = q_n(t\Omega) q_n(-t\Omega)^{-1}.
\]
Furthermore, $\mathcal{P}(q_n(t\Omega))$ commutes with $e^{t\Omega}$.
\end{lemma}
\begin{proof}
Since $\Omega$ is skew-Hermitian, $q_n(t\Omega)^* = q_n(-t\Omega)$.  Hence, by~(\ref{polarformula}), 
\begin{align*}
\mathcal{P}(q_n(t\Omega)) &=  q_n(t\Omega) \left[q_n(-t\Omega)q_n(t\Omega) \right]^{-1/2}.
\end{align*}
This shows that $\mathcal{P}(q_n(t\Omega)) = f(t\Omega)$, where $f(z)=q_n(z)(q_n(-z)q_n(z))^{-1/2}$.  If $q_n(t\Omega)$ is nonsingular, then $f$ is defined on the spectrum of $t\Omega$. Indeed, if $\lambda$ is an eigenvalue of $t\Omega$, then $q_n(\lambda)$, being an eigenvalue of $q_n(t\Omega)$, is nonzero, and $q_n(-\lambda)$, being an eigenvalue of $q_n(-t\Omega)=q_n(t\Omega)^*$, is nonzero.  Thus, $\mathcal{P}(q_n(t\Omega))$ commutes with any function of $t\Omega$ defined on the spectrum of $t\Omega$, including $e^{t\Omega}$.  By similar reasoning, $\left[q_n(-t\Omega)q_n(t\Omega)\right]^{-1/2}$ commutes with $q_n(t\Omega)$, so
\begin{align*}
\mathcal{P}(q_n(t\Omega))^2 
&= q_n(t\Omega) \left[q_n(-t\Omega)q_n(t\Omega) \right]^{-1/2} q_n(t\Omega) \left[q_n(-t\Omega)q_n(t\Omega)\right]^{-1/2} \\
&= q_n(t\Omega)^2 \left[q_n(-t\Omega)q_n(t\Omega)\right]^{-1} \\
&= q_n(t\Omega) q_n(-t\Omega)^{-1}.
\end{align*}
\end{proof}

The preceding lemma implies that if $t$ is sufficiently small, then
\[
(\mathcal{P}(q_n(t\Omega)) + e^{t\Omega}) (\mathcal{P}(q_n(t\Omega))-e^{t\Omega}) =  q_n(t\Omega)q_n(-t\Omega)^{-1} - e^{2t\Omega}.
\]
Using this identity, it is not hard to see that the polynomial $q_n$ for which $\mathcal{P}(q_n(t\Omega))-e^{t\Omega}$ is of the highest order in $t$ is precisely that for which $q_n(t\Omega) q_n(-t\Omega)^{-1} - e^{2t\Omega}$ is of the highest order in $t$.  That polynomial is none other than the numerator in the diagonal Pad\'{e} approximant of $e^{2t\Omega}$, which is precisely $\Theta_n(t\Omega)$~\cite[p. 97]{iserles2000lie}.

\subsection{Exponentiation on the Grassmannian} \label{sec:grassmanexpproof}

We now prove Theorem~\ref{thm:grassmanexp}.  The proof will consist of two parts.  First, we prove the identity~(\ref{PTheta}) by exploiting the block structure of the matrix $Z= \begin{pmatrix} 0 & -K^* \\ K & 0 \end{pmatrix}$.  Then, we insert the right-hand side of~(\ref{PTheta}) into~(\ref{grassmanexpnaive}) and expand the result to obtain Theorem~\ref{thm:grassmanexp}.

Throughout this subsection, we make use of the identities
\begin{equation} \label{Zpowereven}
Z^{2j} = \begin{pmatrix} (-K^*K)^j & 0 \\ 0 & (-KK^*)^j \end{pmatrix}
\end{equation}
and
\begin{equation} \label{Zpowerodd}
Z^{2j+1} = \begin{pmatrix} 0 & -K^*(-KK^*)^j \\ K(-K^*K)^j & 0 \end{pmatrix},
\end{equation}
which hold for for every nonnegative integer $j$.

\begin{lemma} \label{lemma:blockdiag}
Let $r(z) = c_0 + c_1 z + c_2 z^2 + \dots + c_n z^n$ be a polynomial, let $K \in \mathbb{C}^{(m-p) \times p}$, and let $Z = \begin{pmatrix} 0 & -K^* \\ K & 0 \end{pmatrix}$.  Then
\[
r(Z)^* r(Z) = \begin{pmatrix} B^*B & 0 \\ 0 & C^*C \end{pmatrix},
\]
where $B = r(Z)\begin{pmatrix} I \\ 0 \end{pmatrix}$ and $C = r(Z)\begin{pmatrix} 0 \\ I \end{pmatrix}$.
\end{lemma}
\begin{proof}
The diagonal blocks of $r(Z)^* r(Z)$ are automatically given by $B^*B$ and $C^*C$, so it suffices to show that the off-diagonal blocks of $r(Z)^* r(Z)$ vanish.  To this end, observe that the skew-Hermiticity of $Z$ implies $r(Z)^* r(Z) = r(-Z) r(Z)$.  But $r(Z)^* r(Z)$ is Hermitian, so taking the Hermitian part of both sides gives $r(Z)^* r(Z) = \mathrm{sym}\left( r(-Z) r(Z) \right)$.  Since $\mathrm{sym}(Z^j)=0$ for odd $j$, it follows that $r(Z)^*r(Z)$ is a linear combination of even powers of $Z$, all of which are block diagonal by~(\ref{Zpowereven}).
\end{proof}

\begin{lemma} \label{lemma:polarsquare2rect}
Let $r(z) = c_0 + c_1 z + c_2 z^2 + \dots + c_n z^n$ be a polynomial, let $K \in \mathbb{C}^{(m-p) \times p}$, and define $Z = \begin{pmatrix} 0 & -K^* \\ K & 0 \end{pmatrix}$.  If $r(Z)$ has full rank, then
\[
\mathcal{P}(r(Z)) \begin{pmatrix} I \\ 0 \end{pmatrix} = \mathcal{P}\left( r(Z) \begin{pmatrix} I \\ 0 \end{pmatrix} \right).
\]
\end{lemma}
\begin{proof}
In the notation of Lemma~\ref{lemma:blockdiag}, 
\begin{align*}
\mathcal{P}(r(Z)) 
&= r(Z) \left( r(Z)^* r(Z) \right)^{-1/2} \\
&= r(Z) \begin{pmatrix} (B^*B)^{-1/2} & 0 \\ 0 & (C^*C)^{-1/2} \end{pmatrix},
\end{align*}
so
\[
\mathcal{P}(r(Z)) \begin{pmatrix} I \\ 0 \end{pmatrix} = r(Z) \begin{pmatrix} (B^*B)^{-1/2} \\ 0 \end{pmatrix}.
\]
On the other hand,
\begin{align*}
\mathcal{P} \left( r(Z) \begin{pmatrix} I \\ 0 \end{pmatrix} \right) 
&= r(Z) \begin{pmatrix} I \\ 0 \end{pmatrix}  \left( \begin{pmatrix} I & 0 \end{pmatrix} r(Z)^*r(Z) \begin{pmatrix} I \\ 0 \end{pmatrix} \right)^{-1/2} \\
&= r(Z) \begin{pmatrix} I \\ 0 \end{pmatrix} (B^*B)^{-1/2} \\
&= r(Z) \begin{pmatrix} (B^*B)^{-1/2} \\ 0 \end{pmatrix}
\end{align*}
as well.
\end{proof}

The preceding lemma establishes the identity~(\ref{PTheta}).  We now study the quantity $\Theta_n(tZ) \begin{pmatrix} I \\ 0 \end{pmatrix}$ in more detail.

\begin{lemma} \label{lemma:evenodd}
Let $\alpha_n(z)$ and $\beta_n(z)$ be as in Theorem~\ref{thm:grassmanexp}, let $K \in \mathbb{C}^{(m-p) \times p}$, and let $Z = \begin{pmatrix} 0 & -K^* \\ K & 0 \end{pmatrix}$.  Then
\[
\Theta_n(tZ)\begin{pmatrix} I \\ 0 \end{pmatrix} = \begin{pmatrix} \alpha_n(t^2K^* K) \\ tK \beta_n(t^2K^* K) \end{pmatrix}.
\]
\end{lemma}
\begin{proof}
Using~(\ref{Zpowereven}-\ref{Zpowerodd}), we have
\begin{align*}
\Theta_n(tZ) \begin{pmatrix} I \\ 0 \end{pmatrix} 
&= \left(  \sum_{j=0}^{\floor{m/2}} a_{2j} (tZ)^{2j} + \sum_{j=0}^{\floor{(m-1)/2}} a_{2j+1} (tZ)^{2j+1} \right)\begin{pmatrix} I \\ 0 \end{pmatrix} \\
&= \sum_{j=0}^{\floor{m/2}} a_{2j} \begin{pmatrix} (-t^2 K^* K)^j \\ 0 \end{pmatrix} + \sum_{j=0}^{\floor{(m-1)/2}} a_{2j+1} \begin{pmatrix} 0 \\ tK(-t^2 K^*K)^j \end{pmatrix}  \\
&= \begin{pmatrix} \alpha_n(t^2K^*K) \\ 0 \end{pmatrix} + \begin{pmatrix} 0 \\ tK\beta_n(t^2K^*K) \end{pmatrix} \\
&= \begin{pmatrix} \alpha_n(t^2K^* K) \\ tK \beta_n(t^2K^* K) \end{pmatrix}.
\end{align*}
\end{proof}
We are now in a position to prove Theorem~\ref{thm:grassmanexp} by substituting the preceding results into~(\ref{grassmanexpnaive}).  Combining Lemmas~\ref{lemma:polarsquare2rect} and~\ref{lemma:evenodd}, we have
\begin{align*}
\begin{pmatrix} Y & Y_\perp \end{pmatrix} \mathcal{P}(\Theta_n(tZ))\begin{pmatrix} I \\ 0 \end{pmatrix} 
&= \begin{pmatrix} Y & Y_\perp \end{pmatrix} \mathcal{P} \begin{pmatrix}  \alpha_n(t^2K^* K) \\ tK \beta_n(t^2K^* K) \end{pmatrix}\\
&= \mathcal{P} \left( \begin{pmatrix} Y & Y_\perp \end{pmatrix}  \begin{pmatrix}  \alpha_n(t^2K^* K) \\ tK \beta_n(t^2K^* K) \end{pmatrix} \right) \\
&= \mathcal{P} \left( Y \alpha_n(t^2K^* K) + tY_\perp K \beta_n(t^2K^* K) \right) \\
&= \mathcal{P} \left( Y \alpha_n(t^2H^* H) + tH \beta_n(t^2H^* H) \right),
\end{align*}
where the second line follows from~(\ref{polarinvariance}), 
and the last line follows from the fact that $H=Y_\perp K$, and $Y_\perp$ has orthonormal columns.  
This, together with~(\ref{grassmanexpnaive}), completes the proof of Theorem~\ref{thm:grassmanexp}.

\subsection{Exponentiation on the Stiefel Manifold} \label{sec:stiefelexpproof}

We now turn to the proof of Theorem~\ref{thm:stiefelexp}.  Let $q(x,y)$ and $r(x,y)$ be polynomials in non-commutating variables $x$ and $y$, and define
\[
A = Yq\left(t^2H^*H, tY^*H\right) + tHr\left(t^2H^*H,tY^*H\right).
\]
For the moment we assume only that $q(0,0)=1$, but later we will make the additional assumptions~(\ref{polycond1}-\ref{polycond3}) (the first of which implies $q(0,0)=1$).
Using the identities $Y =  \begin{pmatrix} Y & Y_\perp \end{pmatrix} \begin{pmatrix} I \\ 0 \end{pmatrix}$, $H = Y\Omega + Y_\perp K = \begin{pmatrix} Y & Y_\perp \end{pmatrix}  \begin{pmatrix} \Omega \\ K \end{pmatrix}$, $H^*H = K^*K-\Omega^2$, and $Y^*H = \Omega$, we can write
\begin{align*}
A
&= \begin{pmatrix} Y & Y_\perp \end{pmatrix} \left( \begin{pmatrix} I \\ 0 \end{pmatrix} q\left(t^2(K^*K-\Omega^2),t\Omega\right) + \begin{pmatrix} t\Omega \\ tK \end{pmatrix} r\left(t^2(K^*K-\Omega^2,t\Omega\right) \right) \\
&= \begin{pmatrix} Y & Y_\perp \end{pmatrix} \begin{pmatrix} q+t\Omega r \\ tKr \end{pmatrix},
\end{align*}
where we have suppressed the arguments to $q$ and $r$ in the last line to reduce clutter.

Now let $Z = \begin{pmatrix} \Omega & -K^* \\ K & 0 \end{pmatrix}$ and define
\begin{align*}
\widetilde{U} &= \begin{pmatrix} Y & Y_\perp \end{pmatrix} e^{tZ} \begin{pmatrix} I \\ 0 \end{pmatrix} = \mathrm{Exp}^{St}_Y(tH).
\end{align*}
We aim to bound 
\[
\|\mathcal{P}(A)-\widetilde{U}\| = \|\mathcal{P}\left(Yq\left(t^2H^*H, tY^*H\right) + tHr\left(t^2H^*H,tY^*H\right)\right) - \mathrm{Exp}^{St}_Y(tH)\|
\]
using Lemma~\ref{lemma:symmetry}.  Since $A\big|_{t=0}=\begin{pmatrix} Y & Y_\perp \end{pmatrix}$ is unitary, it follows that $\sigma_i(A)=O(1)$ as $t \rightarrow 0$ for each $i=1,2,\dots,p$.  Thus, it is enough to bound $\|\mathrm{skew}(\widetilde{U}^*A)\|$ and $\|(I-\widetilde{U}\widetilde{U}^*)A\|$.  We begin with a lemma.
%We begin by calculating $\widetilde{U}^*A$ and $(I-\widetilde{U}\widetilde{U}^*)A$.

\begin{lemma} \label{lemma:rhsterms}
We have
\begin{align*}
\|\mathrm{skew}(\widetilde{U}^*A)\| &= \left\| \mathrm{skew}\left(\begin{pmatrix} I & 0 \end{pmatrix} e^{-tZ} \begin{pmatrix} q+t\Omega r \\ tKr \end{pmatrix}\right) \right\|, \\
\|(I-\widetilde{U}\widetilde{U}^*)A\| &= \left\| \begin{pmatrix} 0 & 0 \\ 0 & I \end{pmatrix} e^{-tZ} \begin{pmatrix} q+t\Omega r \\ tKr \end{pmatrix} \right\|.
\end{align*}
\end{lemma}
\begin{proof}
The first equality follows from a direct calculation, using the fact that $Z$ is skew-Hermitian and $\begin{pmatrix} Y & Y_\perp \end{pmatrix}$ is unitary.  For the second, observe that
\begin{align*}
(I-\widetilde{U}\widetilde{U}^*)A 
&= \left[ I - \begin{pmatrix} Y & Y_\perp \end{pmatrix} e^{tZ} \begin{pmatrix} I \\ 0 \end{pmatrix} \begin{pmatrix} I & 0 \end{pmatrix} e^{-tZ} \begin{pmatrix} Y^* \\ Y_\perp^* \end{pmatrix} \right] \begin{pmatrix} Y & Y_\perp \end{pmatrix} \begin{pmatrix} q+t\Omega r \\ tKr \end{pmatrix} \\
&= \begin{pmatrix} Y & Y_\perp \end{pmatrix} e^{tZ} \begin{pmatrix} 0 & 0 \\ 0 & I \end{pmatrix} e^{-tZ} \begin{pmatrix} Y^* \\ Y_\perp^* \end{pmatrix} \begin{pmatrix} Y & Y_\perp \end{pmatrix} \begin{pmatrix} q+t\Omega r \\ tKr \end{pmatrix} \\
&= \begin{pmatrix} Y & Y_\perp \end{pmatrix} e^{tZ} \begin{pmatrix} 0 & 0 \\ 0 & I \end{pmatrix} e^{-tZ} \begin{pmatrix} q+t\Omega r \\ tKr \end{pmatrix}.
\end{align*}
The result follows from the fact that the Frobenius norm is unitarily invariant, and $\begin{pmatrix} Y & Y_\perp \end{pmatrix}$ and $e^{tZ}$ are unitary.
\end{proof}

The preceding lemma reveals that the order of accuracy of the approximation~(\ref{stiefelansatz}) can be determined by studying the quantity
\begin{equation} \label{etZA}
e^{-tZ} \begin{pmatrix} q+t\Omega r \\ tKr \end{pmatrix} = e^{-tZ} \begin{pmatrix} q\left(t^2(K^*K-\Omega^2),t\Omega\right) + t\Omega r\left(t^2(K^*K-\Omega^2),t\Omega\right) \\ tK r\left(t^2(K^*K-\Omega^2),t\Omega\right) \end{pmatrix}.
\end{equation}
Let us carry out this task in order to determine, as an illustration, the highest order approximation of the form~(\ref{stiefelansatz}) that can be achieved using polynomials $q(x,y)$ and $r(x,y)$ satisfying~(\ref{polycond1}-\ref{polycond3}) with $n=2$.  The cases $n=1$ and $n=3$ are handled similarly; we leave those details to the reader.  Collectively, these arguments will prove Theorem~\ref{thm:stiefelexp}.

It is a simple exercise to show that when $n=2$, the only polynomials $q(x,y)$ and $r(x,y)$ satisfying~(\ref{polycond1}-\ref{polycond2}) are of the form
\begin{align*}
q(x,y) &= 1-\frac{1}{3}x+cy^2, \\
r(x,y) &= 1-cy,
\end{align*}
where $c$ is a constant.  
Substituting into~(\ref{etZA}), writing $e^{-tZ} = \sum_{k=0}^\infty \frac{(-tZ)^k}{k!}$, and multiplying, one finds after a tedious calculation that
\[
e^{-tZ} \begin{pmatrix} q+t\Omega r \\ tKr \end{pmatrix} = \begin{pmatrix} I + \frac{1}{6} t^2 (K^*K - \Omega^2) - \left(c+\frac{1}{3}\right)t^3 K^*K\Omega \\ -\left(c+\frac{1}{2}\right)t^2 K\Omega \end{pmatrix} + O(t^4). 
\]
Hence, by Lemma~\ref{lemma:rhsterms} and the symmetry of $K^*K$ and $\Omega^2$, we have
\begin{align*}
\|\mathrm{skew}(\widetilde{U}^*A)\| &= \left|c+\frac{1}{3}\right| t^3 \|K^*K\Omega\| + O(t^4),\\
\|(I-\widetilde{U}\widetilde{U}^*)A\| &= \left|c+\frac{1}{2}\right| t^2 \|K\Omega\| + O(t^4).
\end{align*}
The optimal choice of $c$ is $c=-\frac{1}{2}$, giving
\begin{align*}
q(x,y) &= 1-\frac{1}{3}x-\frac{1}{2}y^2 = \gamma_2(x,y), \\
r(x,y) &= 1+\frac{1}{2}y = \delta_2(x,y),
\end{align*}
and
\begin{align*}
\|\mathrm{skew}(\widetilde{U}^*A)\| &= \frac{1}{6} t^3 \|K^*K\Omega\| + O(t^4),\\
\|(I-\widetilde{U}\widetilde{U}^*)A\| &= O(t^4).
\end{align*}
It follows that
\[
\mathrm{Exp}^{St}_Y (tH) = \mathcal{P}\left(Y\gamma_2(t^2 H^*H, tY^*H) + tH\delta_2(t^2 H^*H, tY^*H)\right) + O(t^3).
\]
Clearly, no other choice of $c$ will improve this approximant's order of accuracy, proving~(\ref{stiefelerror}) for $n=2$.

If it happens that $Y^*H=\Omega=0$, then~(\ref{grassmanexp}) and~(\ref{stiefelexp}) coincide, and~(\ref{polycond1}-\ref{polycond2}) and Theorem~\ref{thm:grassmanexp} imply
\begin{align*}
\mathcal{P}\left(Y\gamma_2(t^2 H^*H,0) + tH\delta_2(t^2 H^*H, 0)\right)
&= \mathcal{P}\left(Y\alpha_2(t^2 H^*H) + tH\beta_2(t^2 H^*H)\right) \\
&= \mathrm{Exp}^{Gr}_Y (tH) + O(t^5) \\
&= \mathrm{Exp}^{St}_Y (tH) + O(t^5).
\end{align*}
Likewise, if $m=p$, so that $H=Y\Omega$, $YY^*=I$, and $\mathrm{Exp}^{St}_Y (tH)=Ye^{t\Omega}$, then~(\ref{polycond3}),~(\ref{polarinvariance}), and Theorem~\ref{thm:polarexp} imply
\begin{align*}
\mathcal{P}\left(Y\gamma_2(t^2 H^*H,tY^*H) + tH\delta_2(t^2 H^*H,  \right.&\left. tY^*H)\right) \\
&= \mathcal{P}\left(Y\gamma_2(-t^2 \Omega^2,t\Omega) + tY\Omega\delta_2(-t^2 \Omega^2, t\Omega)\right) \\
&= \mathcal{P}\left(Y \Theta_2(t\Omega)\right) \\
&= Y\mathcal{P}\left(\Theta_2(t\Omega)\right) \\
&= Ye^{t\Omega} + O(t^5) \\
&= \mathrm{Exp}^{St}_Y (tH) + O(t^5).
\end{align*}
These observations prove~(\ref{stiefelapprox}-\ref{stiefelapprox2}) for the case $n=2$.  The proof of Theorem~\ref{thm:stiefelexp} is completed by performing analogous arguments for the cases $n=1$ and $n=3$.

\subsection{Geometric and Arithmetic Means of Unitary Matrices} \label{sec:geoarithproof}

We now prove Proposition~\ref{thm:polargeodesic} and Theorem~\ref{thm:geoarith}.

\paragraph{Proof of Proposition~\ref{thm:polargeodesic}}
Without loss of generality, consider the case in which $U_1 = I$, so that
\begin{align*}
(1-s)U_1 + s U_2
&= (1-s)I + s e^{t\Omega}.
\end{align*}
By Lemma~\ref{lemma:symmetrysimple}, it suffices to examine the norm of
\[
\mathrm{skew}\left(e^{-st\Omega} \left((1-s)I + s e^{t\Omega}\right) \right).
\]
The series expansion of $e^{-st\Omega} \left((1-s)I + s e^{t\Omega}\right)$ reads
\[
e^{-st\Omega} \left((1-s)I + s e^{t\Omega}\right) = \sum_{k=0}^\infty \frac{(1-s)s^k (-1)^k + s(1-s)^k}{k!} (t\Omega)^k.
\]
Since $\Omega$ is skew-Hermitian, $\mathrm{skew}(\Omega^k) = 0$ for every even $k$, showing that
\[
\mathrm{skew} \left(e^{-st\Omega} \left((1-s)I + s e^{t\Omega}\right)\right) = \sum_{\substack{k=1,\\k \text{ odd }}}^\infty \frac{(1-s)s^k (-1)^k + s(1-s)^k}{k!} (t\Omega)^k.
\]
When $s=1/2$, each term in the series vanishes, giving 
\[
\mathrm{skew} \left(e^{-st\Omega} \left((1-s)I + s e^{t\Omega}\right)\right)\big|_{s=1/2} = 0.
\]
When $s \neq 1/2$, the first non-vanishing term is of order $t^3$, showing that 
\[
\mathrm{skew} \left(e^{-st\Omega} \left((1-s)I + s e^{t\Omega}\right) \right) = O(t^3).
\]
The result follows by applying Lemma~\ref{lemma:symmetrysimple}.

\paragraph{Proof of Theorem~\ref{thm:geoarith}}
Let 
\[
A(t) = \sum_{i=1}^n w_i U_i(t)
\] 
and 
\[
\widetilde{U}(t) = \mathbb{G}(U_1(t),\dots,U_n(t);w).
\]
Observe that if $A(t)=U(t)H(t)$ is the polar decomposition of $A(t)$, then, by~(\ref{arithformula}),
\[
U(t) = \mathbb{A}(U_1(t),\dots,U_n(t);w).
\]
Moreover, using the fact that $\sum_{i=1}^n w_i = 1$, we have
\[
A(t) = U_1(t) + \sum_{i=2}^n w_i (U_i(t)-U_1(t)).
\]
This shows, by~(\ref{distanceassumption}), that 
\[
\lim_{t\rightarrow 0} A(t) = \lim_{t\rightarrow 0} U_1(t) = U_1(0).
\]
The latter matrix is unitary, so $\|A(t)\| = O(1)$. This is independent of the $U_i(t)$ we chose to pull out of the sum, since (\ref{distanceassumption}) implies that $U_1(0)=U_2(0)=\dots=U_n(0)$.

Now let $\Omega_i(t) = \frac{1}{t}\log(\widetilde{U}(t)^* U_i(t))$ for each $i$, so that
\begin{align*}
\widetilde{U}(t)^* U_i(t)
&= e^{t\Omega_i(t)}.
\end{align*}
Note that $\Omega_i(t)=O(1)$ by~(\ref{distanceassumption}).
In addition, by~(\ref{geomeanEL}),
\[
\sum_{i=1}^n w_i \Omega_i(t) = 0.
\] 
Suppressing the dependencies on $t$ for ease of reading, it follows that
\begin{align*}
\widetilde{U}^* A
&= \sum_{i=1}^m w_i \widetilde{U}^* U_i \\
&= \sum_{i=1}^m w_i e^{t\Omega_i} \\
&= \sum_{i=1}^m w_i I + \sum_{i=1}^m w_i t\Omega_i + \sum_{i=1}^m w_i \left( e^{t\Omega_i} - I - t\Omega_i \right) \\
&= I + \sum_{i=1}^m w_i \left( e^{t\Omega_i} - I - t\Omega_i \right).
\end{align*}
The skew-Hermitian part of $\widetilde{U}^*A$ is thus given by
\[
\mathrm{skew}(\widetilde{U}^* A) =  \sum_{i=1}^m w_i \left( \frac{e^{t\Omega_i} - e^{-t\Omega_i}}{2} - t\Omega_i \right).
\]
Since
\[
\frac{e^{t\Omega_i} - e^{-t\Omega_i}}{2} - t\Omega_i = O(t^3)
\] 
and $\|A\|=O(1)$, it follows from Lemma~\ref{lemma:symmetrysimple} that
\[
U-\widetilde{U} = O(t^3),
\]
i.e.,
\[
\mathbb{A}(U_1(t),\dots,U_n(t);w) = \mathbb{G}(U_1(t),\dots,U_n(t);w) + O(t^3).
\]

\section{Numerical Examples} \label{sec:numerical}

In this section, we discuss how the projected polynomials proposed in Theorems~\ref{thm:polarexp},~\ref{thm:grassmanexp}, and~\ref{thm:stiefelexp} can be efficiently computed, focusing on iterative methods for computing the polar decomposition. We then present numerical examples that illustrate their order of accuracy.

\subsection{Iterative Methods for Computing the Polar Decomposition}

The cost of computing a projected polynomial is largely dominated by the cost of evaluating the map $\mathcal{P}$.  
This map can be computed efficiently via a number of different iterative methods.  The most widely known, applicable when $m=p$, is the Newton iteration
\begin{equation} \label{Newton}
X_{k+1} = \frac{1}{2}(X_k + X_k^{-*}), \quad X_0 = A.
\end{equation}
The iterates $X_k$ so defined converge quadratically to the unitary factor $\mathcal{P}(A)=U$ in the polar decomposition $A=UH$ for any nonsingular square matrix $A$~\cite[Theorem 8.12]{higham2008functions}.  A closely related iteration, applicable when $m \ge p$, is given by
\begin{equation} \label{Newtonvariant}
X_{k+1} = 2X_k(I+X_k^*X_k)^{-1}, \quad X_0 = A.
\end{equation}
These iterates converge quadratically to $\mathcal{P}(A)$ for any full-rank $A \in \mathbb{C}^{m \times p}$ ($m \ge p$)~\cite[Corollary 8.14(b)]{higham2008functions}.  Finally, the Newton-Schulz iteration
\begin{equation} \label{NewtonSchulz}
X_{k+1} = \frac{1}{2} X_k (I - 3 X_k^* X_k), \quad X_0=A
\end{equation}
provides an inverse-free iteration whose iterates $X_k$ converge quadratically to $\mathcal{P}(A)$ for any $A \in \mathbb{C}^{m \times p}$ ($m \ge p$) whose singular values all lie in the interval $(0,\sqrt{3})$~\cite[Problem 8.20]{higham2008functions}.  For further information, including other iterations for computing $\mathcal{P}(A)$, see~\cite[Chapter 8]{higham2008functions}.

\subsection{Numerical Convergence}

We tested the accuracy of the projected polynomials detailed in Theorems~\ref{thm:polarexp},~\ref{thm:grassmanexp}, and~\ref{thm:stiefelexp} by applying them to randomly generated inputs.  To calculate $\mathcal{P}$, we used~(\ref{Newton}) for square matrices and~(\ref{NewtonSchulz}) for rectangular matrices.   The results of the tests, detailed in Tables~\ref{tab:unitary}-\ref{tab:stiefel}, corroborate the convergence rates predicted by the theory.

\begin{table}[t]
\centering
\begin{filecontents*}{Data/unitary.dat}
   1.0000000e+00   1.6068041e+00   6.9450139e-02   1.3117474e-03
   2.0000000e+00   2.4327796e-01   2.4443680e-03   1.1085786e-05
   4.0000000e+00   3.2229664e-02   7.8599869e-05   8.8300161e-08
   8.0000000e+00   4.0912793e-03   2.4735871e-06   6.9317949e-10
\end{filecontents*}
\pgfplotstabletypeset[
every head row/.style={after row=\midrule,before row={\midrule & \multicolumn{3}{c|}{$n=1$} & \multicolumn{3}{c|}{$n=2$} & \multicolumn{3}{c}{$n=3$} \\ \midrule}},
create on use/rate1/.style={create col/dyadic refinement rate={1}},
create on use/rate2/.style={create col/dyadic refinement rate={2}},
create on use/rate3/.style={create col/dyadic refinement rate={3}},
columns={0,1,rate1,2,rate2,3,rate3},
columns/0/.style={sci zerofill,column type/.add={}{|},column name={$t_0/t$}},
columns/1/.style={dec sep align={c|},sci,sci zerofill,precision=3,column type/.add={}{|},column name={Error}},
columns/2/.style={dec sep align={c|},sci,sci zerofill,precision=3,column type/.add={}{|},column name={Error}}, 
columns/3/.style={dec sep align={c|},sci,sci zerofill,precision=3,column type/.add={}{|},column name={Error}},
columns/rate1/.style={fixed zerofill,precision=3,column type/.add={}{|},column name={Order}},
columns/rate2/.style={fixed zerofill,precision=3,column type/.add={}{|},column name={Order}},
columns/rate3/.style={fixed zerofill,precision=3,column name={Order}}
]
{Data/unitary.dat}
\caption{Errors in approximating the exponential of a skew-Hermitian matrix $\Omega$ with the projected polynomials of Theorem~\ref{thm:polarexp}.  Shown above are the errors $\|\mathcal{P}(\Theta_n(t\Omega))-e^{t\Omega}\|$ versus $t$ for $n=1,2,3$, where $t_0=0.01$ and $\Omega$ is a random $1000 \times 1000$ skew-Hermitian matrix.}
\label{tab:unitary}
\vspace{-0.2in}
\end{table}

\begin{table}[t]
\centering
\begin{filecontents*}{Data/grassman.dat}
   1.0000000e+00   5.0295672e-01   9.3834048e-03   7.8494991e-05
   2.0000000e+00   6.9506631e-02   3.0900614e-04   6.3518862e-07
   4.0000000e+00   8.9326573e-03   9.7806844e-06   5.0059481e-09
   8.0000000e+00   1.1245699e-03   3.0661775e-07   3.9194400e-11
\end{filecontents*}
\pgfplotstabletypeset[
every head row/.style={after row=\midrule,before row={\midrule & \multicolumn{3}{c|}{$n=1$} & \multicolumn{3}{c|}{$n=2$} & \multicolumn{3}{c}{$n=3$} \\ \midrule}},
create on use/rate1/.style={create col/dyadic refinement rate={1}},
create on use/rate2/.style={create col/dyadic refinement rate={2}},
create on use/rate3/.style={create col/dyadic refinement rate={3}},
columns={0,1,rate1,2,rate2,3,rate3},
columns/0/.style={sci zerofill,column type/.add={}{|},column name={$t_0/t$}},
columns/1/.style={dec sep align={c|},sci,sci zerofill,precision=3,column type/.add={}{|},column name={Error}},
columns/2/.style={dec sep align={c|},sci,sci zerofill,precision=3,column type/.add={}{|},column name={Error}}, 
columns/3/.style={dec sep align={c|},sci,sci zerofill,precision=3,column type/.add={}{|},column name={Error}},
columns/rate1/.style={fixed zerofill,precision=3,column type/.add={}{|},column name={Order}},
columns/rate2/.style={fixed zerofill,precision=3,column type/.add={}{|},column name={Order}},
columns/rate3/.style={fixed zerofill,precision=3,column name={Order}}
]
{Data/grassman.dat}
\caption{Errors in approximating the Riemannian exponential map on the Grassmannian manifold with the projected polynomials of Theorem~\ref{thm:grassmanexp}.  Shown above are the errors $\|\mathcal{P}(Y\alpha_n(t^2H^*H)+tH\beta_n(t^2H^*H))-\mathrm{Exp}^{Gr}_Y(tH)\|$ versus $t$ for $n=1,2,3$, where $t_0=0.01$, $Y$ is a random $2000 \times 400$ matrix with orthonormal columns, and $H$ is a random $2000 \times 400$ matrix satisfying $Y^*H=0$.}
\label{tab:grassman}
\vspace{-0.35in}
\end{table}

\begin{table}[t]
\centering
\begin{filecontents*}{Data/stiefel.dat}
   1.0000000e+00   1.3360080e+00   1.4161102e-01   1.6543648e-02
   2.0000000e+00   3.0130724e-01   1.9308357e-02   9.7236181e-04
   4.0000000e+00   7.1953552e-02   2.4786417e-03   5.9303970e-05
   8.0000000e+00   1.7744606e-02   3.1198462e-04   3.6811235e-06
\end{filecontents*}
\pgfplotstabletypeset[
every head row/.style={after row=\midrule,before row={\midrule & \multicolumn{3}{c|}{$n=1$} & \multicolumn{3}{c|}{$n=2$} & \multicolumn{3}{c}{$n=3$} \\ \midrule}},
create on use/rate1/.style={create col/dyadic refinement rate={1}},
create on use/rate2/.style={create col/dyadic refinement rate={2}},
create on use/rate3/.style={create col/dyadic refinement rate={3}},
columns={0,1,rate1,2,rate2,3,rate3},
columns/0/.style={sci zerofill,column type/.add={}{|},column name={$t_0/t$}},
columns/1/.style={dec sep align={c|},sci,sci zerofill,precision=3,column type/.add={}{|},column name={Error}},
columns/2/.style={dec sep align={c|},sci,sci zerofill,precision=3,column type/.add={}{|},column name={Error}}, 
columns/3/.style={dec sep align={c|},sci,sci zerofill,precision=3,column type/.add={}{|},column name={Error}},
columns/rate1/.style={fixed zerofill,precision=3,column type/.add={}{|},column name={Order}},
columns/rate2/.style={fixed zerofill,precision=3,column type/.add={}{|},column name={Order}},
columns/rate3/.style={fixed zerofill,precision=3,column name={Order}}
]
{Data/stiefel.dat}
\caption{Errors in approximating the Riemannian exponential map on the Stiefel manifold with the projected polynomials of Theorem~\ref{thm:stiefelexp}.  Shown above are the errors $\|\mathcal{P}(Y\gamma_n(t^2H^*H,tY^*H)+tH\delta_n(t^2H^*H,tY^*H))-\mathrm{Exp}^{St}_Y(tH)\|$ versus $t$ for $n=1,2,3$, where $t_0=0.01$, $Y$ is a random $2000 \times 400$ matrix with orthonormal columns, and $H$ is a random $2000 \times 400$ matrix satisfying $Y^*H=-H^*Y$.}
\label{tab:stiefel}
\vspace{-0.4in}
\end{table}

\section{Conclusion}

This paper has presented a family of high-order retractions on the unitary group, the Grassmannian manifold, and the Stiefel manifold.  All of these retractions were constructed by projecting certain matrix polynomials onto the set of matrices with orthonormal columns using the polar decomposition, or, in the case of the Grassmannian, using either the polar decomposition or the QR decomposition.  There are several interesting applications and extensions of this strategy that seem worthwhile to pursue.  On quadratic Lie groups other than the unitary group, one might consider adopting the same strategy, replacing the polar decomposition with the generalized polar decomposition~\cite{munthe2001generalized,higham2010canonical}.  It might also be worthwhile to consider projecting rational functions, rather than polynomials, to achieve higher accuracy for comparable cost.  It may also be possible to leverage these retractions, together with methods for computing their derivatives~\cite{gawlik2016iterative}, to construct high-order approximations of parallel transport operators on matrix manifolds; see~\cite[Section 8.1.2]{absil2009optimization}.

It is worth noting that many of the constructions in this paper might generalize nicely to infinite dimensions.  For instance, replacing the matrices $Y$ and $H$ in Theorem~\ref{thm:grassmanexp} with quasi-matrices in the sense of~\cite{townsend2015continuous}, one obtains a method for approximating geodesics between finite-dimensional function spaces, with $H^*H$ playing the role of a Gramian, and with $\mathcal{P}$ interpreted as the map sending an ordered basis of functions to the nearest ordered, orthonormal basis of functions.

\section{Acknowledgements}

We wish to thank the developers of the non-commutative algebra package \emph{NCAlgebra}~\cite{NCAlgebra}, which we used to carry out some of the calculations that appeared/were mentioned in Section~\ref{sec:stiefelexpproof}. The first author is supported by NSF grants CMMI-1334759, DMS-1345013, and DMS-1703719. The second author is supported by NSF grants CMMI-1334759, DMS-1345013, DMS-1411792.

%\printbibliography
\bibliography{references}

\begin{thebibliography}{10}

\bibitem{absil2004riemannian}
{\sc P.-A. Absil, R.~Mahony, and R.~Sepulchre}, {\em Riemannian geometry of
  grassmann manifolds with a view on algorithmic computation}, Acta Applicandae
  Mathematica, 80 (2004), pp.~199--220.

\bibitem{absil2009optimization}
{\sc P.-A. Absil, R.~Mahony, and R.~Sepulchre}, {\em Optimization Algorithms on
  Matrix Manifolds}, Princeton University Press, 2009.

\bibitem{absil2012projection}
{\sc P.-A. Absil and J.~Malick}, {\em Projection-like retractions on matrix
  manifolds}, SIAM Journal on Optimization, 22 (2012), pp.~135--158.

\bibitem{adler2002newton}
{\sc R.~L. Adler, J.-P. Dedieu, J.~Y. Margulies, M.~Martens, and M.~Shub}, {\em
  Newton's method on riemannian manifolds and a geometric model for the human
  spine}, IMA Journal of Numerical Analysis, 22 (2002), pp.~359--390.

\bibitem{bou2009hamilton}
{\sc N.~Bou-Rabee and J.~E. Marsden}, {\em Hamilton--{P}ontryagin integrators
  on {L}ie groups part i: Introduction and structure-preserving properties},
  Foundations of Computational Mathematics, 9 (2009), pp.~197--219.

\bibitem{celledoni2000approximating}
{\sc E.~Celledoni and A.~Iserles}, {\em Approximating the exponential from a
  {L}ie algebra to a {L}ie group}, Mathematics of Computation, 69 (2000),
  pp.~1457--1480.

\bibitem{celledoni2001methods}
{\sc E.~Celledoni and A.~Iserles}, {\em Methods for the approximation of the
  matrix exponential in a {L}ie-algebraic setting}, IMA Journal of Numerical
  Analysis, 21 (2001), pp.~463--488.

\bibitem{edelman1998geometry}
{\sc A.~Edelman, T.~A. Arias, and S.~T. Smith}, {\em The geometry of algorithms
  with orthogonality constraints}, SIAM Journal on Matrix Analysis and
  Applications, 20 (1998), pp.~303--353.

\bibitem{fan1955some}
{\sc K.~Fan and A.~J. Hoffman}, {\em Some metric inequalities in the space of
  matrices}, Proceedings of the American Mathematical Society, 6 (1955),
  pp.~111--116.

\bibitem{fiori2015tangent}
{\sc S.~Fiori, T.~Kaneko, and T.~Tanaka}, {\em Tangent-bundle maps on the
  grassmann manifold: Application to empirical arithmetic averaging}, IEEE
  Transactions on Signal Processing, 63 (2015), pp.~155--168.

\bibitem{gallivan2003efficient}
{\sc K.~A. Gallivan, A.~Srivastava, X.~Liu, and P.~Van~Dooren}, {\em Efficient
  algorithms for inferences on {G}rassmann manifolds}, in 2003 IEEE Workshop on
  Statistical Signal Processing, IEEE, 2003, pp.~315--318.

\bibitem{gawlik2016embedding}
{\sc E.~S. Gawlik and M.~Leok}, {\em Embedding-based interpolation on the
  special orthogonal group}, (Preprint),  (2016).

\bibitem{gawlik2016interpolation}
{\sc E.~S. Gawlik and M.~Leok}, {\em Interpolation on symmetric spaces via the
  generalized polar decomposition}, (Preprint),  (2016).

\bibitem{gawlik2016iterative}
{\sc E.~S. Gawlik and M.~Leok}, {\em Iterative computation of the {F}r\'{e}chet
  derivative of the polar decomposition}, (Preprint),  (2016).

\bibitem{gawlik2011geometric}
{\sc E.~S. Gawlik, P.~Mullen, D.~Pavlov, J.~E. Marsden, and M.~Desbrun}, {\em
  Geometric, variational discretization of continuum theories}, Physica D:
  Nonlinear Phenomena, 240 (2011), pp.~1724--1760.

\bibitem{grohs2013quasi}
{\sc P.~Grohs}, {\em Quasi-interpolation in riemannian manifolds}, IMA Journal
  of Numerical Analysis, 33 (2013), pp.~849--874.

\bibitem{hairer2006geometric}
{\sc E.~Hairer, C.~Lubich, and G.~Wanner}, {\em Geometric Numerical
  Integration: Structure-Preserving Algorithms for Ordinary Differential
  Equations}, vol.~31, Springer Science \& Business Media, 2006.

\bibitem{hall2014lie}
{\sc J.~Hall and M.~Leok}, {\em Lie group spectral variational integrators},
  Foundations of Computational Mathematics, 17 (2017), pp.~199--257.

\bibitem{NCAlgebra}
{\sc J.~W. Helton, M.~C. de~Oliveira, B.~Miller, and M.~Stankus}, {\em The
  ncalgebra suite - version 5.0}.
\newblock \url{http://math.ucsd.edu/~ncalg/}, 2017.

\bibitem{higham2008functions}
{\sc N.~J. Higham}, {\em Functions of Matrices: Theory and Computation}, SIAM,
  2008.

\bibitem{higham2005functions}
{\sc N.~J. Higham, D.~S. Mackey, N.~Mackey, and F.~Tisseur}, {\em Functions
  preserving matrix groups and iterations for the matrix square root}, SIAM
  Journal on Matrix Analysis and Applications, 26 (2005), pp.~849--877.

\bibitem{higham2010canonical}
{\sc N.~J. Higham, C.~Mehl, and F.~Tisseur}, {\em The canonical generalized
  polar decomposition}, SIAM Journal on Matrix Analysis and Applications, 31
  (2010), pp.~2163--2180.

\bibitem{iserles2000lie}
{\sc A.~Iserles, H.~Z. Munthe-Kaas, S.~P. N{\o}rsett, and A.~Zanna}, {\em
  Lie-group methods}, Acta Numerica 2000, 9 (2000), pp.~215--365.

\bibitem{iserles2005efficient}
{\sc A.~Iserles and A.~Zanna}, {\em Efficient computation of the matrix
  exponential by generalized polar decompositions}, SIAM Journal on Numerical
  Analysis, 42 (2005), pp.~2218--2256.

\bibitem{kaneko2013empirical}
{\sc T.~Kaneko, S.~Fiori, and T.~Tanaka}, {\em Empirical arithmetic averaging
  over the compact stiefel manifold}, IEEE Transactions on Signal Processing,
  61 (2013), pp.~883--894.

\bibitem{kobilarov2009lie}
{\sc M.~Kobilarov, K.~Crane, and M.~Desbrun}, {\em Lie group integrators for
  animation and control of vehicles}, ACM Transactions on Graphics (TOG), 28
  (2009), p.~16.

\bibitem{krall1949new}
{\sc H.~L. Krall and O.~Frink}, {\em A new class of orthogonal polynomials:
  {T}he {B}essel polynomials}, Transactions of the American Mathematical
  Society, 65 (1949), pp.~100--115.

\bibitem{li2002perturbation}
{\sc W.~Li and W.~Sun}, {\em Perturbation bounds of unitary and subunitary
  polar factors}, SIAM Journal on Matrix Analysis and Applications, 23 (2002),
  pp.~1183--1193.

\bibitem{lui2012advances}
{\sc Y.~M. Lui}, {\em Advances in matrix manifolds for computer vision}, Image
  and Vision Computing, 30 (2012), pp.~380--388.

\bibitem{lundstrom2002adaptive}
{\sc E.~Lundstr{\"o}m and L.~Eld{\'e}n}, {\em Adaptive eigenvalue computations
  using newton's method on the grassmann manifold}, SIAM Journal on Matrix
  Analysis and Applications, 23 (2002), pp.~819--839.

\bibitem{moakher2002means}
{\sc M.~Moakher}, {\em Means and averaging in the group of rotations}, SIAM
  Journal on Matrix Analysis and Applications, 24 (2002), pp.~1--16.

\bibitem{munthe2001generalized}
{\sc H.~Z. Munthe-Kaas, G.~Quispel, and A.~Zanna}, {\em Generalized polar
  decompositions on lie groups with involutive automorphisms}, Foundations of
  Computational Mathematics, 1 (2001), pp.~297--324.

\bibitem{sander2012geodesic}
{\sc O.~Sander}, {\em Geodesic finite elements on simplicial grids},
  International Journal for Numerical Methods in Engineering, 92 (2012),
  pp.~999--1025.

\bibitem{sander2015geodesic}
{\sc O.~Sander}, {\em Geodesic finite elements of higher order}, IMA J. Numer.
  Anal., 36 (2016), pp.~238--266.

\bibitem{townsend2015continuous}
{\sc A.~Townsend and L.~N. Trefethen}, {\em Continuous analogues of matrix
  factorizations}, Proc. R. Soc. A, 471 (2015), p.~20140585.

\bibitem{turaga2011statistical}
{\sc P.~Turaga, A.~Veeraraghavan, A.~Srivastava, and R.~Chellappa}, {\em
  Statistical computations on grassmann and stiefel manifolds for image and
  video-based recognition}, Pattern Analysis and Machine Intelligence, IEEE
  Transactions on, 33 (2011), pp.~2273--2286.

\bibitem{yoshida1990construction}
{\sc H.~Yoshida}, {\em Construction of higher order symplectic integrators},
  Physics Letters A, 150 (1990), pp.~262--268.

\end{thebibliography}

\end{document}